\documentclass[a4paper,11pt]{amsart}

\usepackage{amssymb}
\usepackage{placeins}
\usepackage{a4wide}
\newtheorem{thm}{Theorem}%[section]
\newtheorem{lem}[thm]{Lemma}
\newtheorem {cor}[thm]{Corollary}
\newtheorem {prop}[thm]{Proposition}

\theoremstyle{definition}

\theoremstyle{remark}
\newtheorem{rem}[thm]{Remark}

\DeclareMathOperator{\ord}{ord}
\DeclareMathOperator{\Gal}{Gal}

\DeclareMathOperator{\id}{id}

\DeclareMathOperator{\rad}{rad}

\DeclareMathOperator{\N}{N}
\DeclareMathOperator{\Li}{Li}
\DeclareMathOperator{\ind}{ind}
\DeclareMathOperator{\re}{Re}
\DeclareMathOperator{\im}{Im}

\newcommand{\Q}{\mathbb{Q}}

\newcommand{\Z}{\mathbb{Z}}

\newcommand{\p}{\mathfrak{p}}

\newcommand{\OO}{\mathcal{O}}
\newcommand{\mc}{\mathcal}
\renewcommand{\geq}{\geqslant}
\renewcommand{\leq}{\leqslant}

\newcommand{\group}[1]{\langle#1\rangle}
\newcommand{\av}[1]{\left\lvert#1\right\rvert}

\newcommand{\bo}[1]{O\left( #1 \right)}
\newcommand{\cset}[1]{\left\lvert\left\{ #1 \right\}\right\rvert}
\newcommand{\set}[1]{\left\{ #1 \right\}}

\usepackage{xr-hyper}
\usepackage{hyperref}

\begin{document}

\title[Divisibility conditions on the order]{Divisibility conditions on the order\\ of the reductions of algebraic numbers}

% \author[short version for running head]{name for top of paper}
\author{Pietro~Sgobba}
\address{Department of Mathematics, University of Luxembourg, 6 Avenue de la Fonte, 4364 Esch-sur-Alzette, Luxembourg}
%\curraddr{}
\email{pietrosgobba1@gmail.com}
%\thanks{}

%    \subjclass is required.
%\date{}

%\dedicatory{}

%    Abstract is required.
\begin{abstract}
Let $K$ be a number field, and let $G$ be a finitely generated  subgroup of $K^\times$. Without relying on the Generalized Riemann Hypothesis we prove an asymptotic formula for the number of primes $\mathfrak p$ of $K$ such that the order of $(G\bmod\mathfrak p)$ is divisible by a fixed integer. We also provide a rational expression for the natural density of this set. Furthermore, we study the primes $\mathfrak p$ for which the order is $k$-free, and those for which the order has a prescribed $\ell$-adic valuation for finitely many primes $\ell$. An additional condition on the Frobenius conjugacy class of $\mathfrak p$ may be considered. In order to establish these results, we prove an unconditional version of the Chebotarev density theorem for Kummer extensions of number fields.
\end{abstract}

\maketitle

%    Text of article.

\section{Introduction}
Consider a number field $K$ and let $G$ be a finitely generated subgroup of $K^\times$. If $\p$ is a prime of $K$ such that $v_\p(g)=0$ for all $g\in G$, then the reduction $(G\bmod\p)$ is a well-defined subgroup of $k_\p^\times$, where $k_\p$ is the residue field at $\p$ and $v_\p$ the $\p$-adic valuation over $K$. In this paper we investigate the set consisting of the primes $\p$ of $K$ such that the order of $(G\bmod\p)$ is well-defined and it satisfies some divisibility conditions.

More precisely, denote by $\ord_\p(G)$ the order of  $(G\bmod\p)$. In Theorem \ref{thm-divisible} we prove an asymptotic formula for the number of primes $\p$ such that $m\mid\ord_\p(G)$, where $m$ is any given positive integer. We also consider the primes $\p$ such that $\ord_\p(G)$ is $k$-free, i.e.\ it is not divisible by $k$-th powers (greater than $1$), where $k\geq2$. In Theorem \ref{thm-kfree}, relying on the previous result, we prove an asymptotic formula for the number of primes satisfying this condition. 
Given a finite Galois extension of $K$, an additional condition on the conjugacy class of the Frobenius automorphisms of the primes lying above $\p$ may also be considered.
Notice that in this paper we do not rely on the Generalized Riemann Hypothesis (GRH). In fact, in Theorem \ref{thm-chebTZ-kummer}, which will be essential in the proof of Theorem \ref{thm-divisible}, we give an unconditional version of the Chebotarev density theorem for cyclotomic-Kummer extensions of number fields. 

The mathematical questions addressed in this paper are closely related to Artin's Conjecture on primitive roots, and hence are part of an active research area; see Moree's survey \cite{MoreeArtin}. The density of rational primes $p$ such that $m\mid\ord_p(g)$, where $g\in\Q^\times\setminus\{\pm1\}$, has been recently studied by Pappalardi \cite{pappa-sf, pappa} (also replacing $g$ with a group of rational numbers), by Moree \cite{moree}, and previously by Wiertelak \cite{wiert}. Our paper provides generalizations and variants of various results by Pappalardi and Moree, as described in the next sections. Over a number field, Debry and Perucca considered the density of the primes $\p$ such that $\ord_\p(G)$, where $G$ is a group consisting of algebraic numbers, is not divisible by some fixed prime number (and described how this permits to treat general divisibility conditions); see \cite{Debry,PeruccaWIN}. Under the assumption of GRH, more general results over number fields hold, e.g.\ for $\ord_\p(G)$ satisfying a given modular congruence; see \cite{zieg} by Ziegler and \cite{PS1} by Perucca and the author. For more references and historical background we refer to \cite[Sect.\ 9.2 and 9.3]{MoreeArtin}.

\subsection{Notation}\label{not1}
If $m\geq1$ is an integer, then we denote by $\zeta_m$ a primitive $m$-th root of unity, and by  $m^\infty$ the supernatural number $\prod_{p\mid m \text{ prime}}p^\infty$. As customary, $\mu$ is the M\"obius function.

We fix an algebraic closure $\overline K$ of $K$.
For $m,n\geq1$ with $n\mid m$, we write $K_{m,n}:=K(\zeta_m,G^{1/n})$ for the $n$-th Kummer extension related to $G$ over $K(\zeta_m)$, i.e.\ the subextension of $\overline K/K(\zeta_m)$ obtained by adding the $n$-th roots of all elements in $G$. If $F/K$ is a finite Galois extension, and $\p$ is a prime of $K$ which does not ramify in $F$, then we denote by $(\p,F/K)$ the conjugacy class of the Frobenius elements at the primes of $F$ above $\p$. If $\mc S$ is a set of primes of $K$, then we let $\pi_{\mc S}(x)$ be the number of primes in $\mc S$ with norm up to $x$.

\subsection{Outline of the main results}
The main result is the following.
\begin{thm}\label{thm-divisible}
Let $F/K$ be a finite Galois extension of number fields, and let $C$ be a conjugacy-stable subset of $\Gal(F/K)$. Let $G$ be a finitely generated and  torsion-free subgroup of $K^\times$ of positive rank $r$. Let $m$ be a positive integer, and consider the set of primes of $K$ given by
\[ \mc P_m =\set{ \p: m\mid \ord_\p(G),\,\begin{pmatrix}\p \\ F/K  \end{pmatrix}\subseteq C }  \]
(where we are tacitly discarding the finitely many primes $\p$ that ramify in $F$ or such that $v_\p(g)\neq0$ for some $g\in G$).
Then we have
\begin{equation}\label{eq-divisible}
\pi_{\mc P_m}(x)=\frac{x}{\log x}\varrho_{C,m} +O\left( x\bigg( \frac{(\log\log x)^2}{\log x} \bigg)^{1+\frac{1}{r+2}} \right) ,
\end{equation}
where
\begin{equation}\label{eq-div-dens}
\varrho_{C,m}:=\sum_{n\mid m^\infty}\,\sum_{d\mid m}\frac{\mu(d)c(mn,dn)}{[F_{mn,dn}:K]}
\end{equation}
and where, for all positive integers $a,b$ with $b\mid a$, we set
\begin{equation}\label{coeff-c1}
c(a,b):=| C\cap\Gal(F/F\cap K_{a,b}) | .
\end{equation}
The constant implied by the $O$-term depends only on  $F$, $K$, $G$.
\end{thm}

The assumption that $G$ is torsion-free allows some simplifications in the proofs, and in Remark \ref{rem-torsion} we explain how to deal with the general case. Also notice that the series in \eqref{eq-div-dens} is  convergent by Proposition \ref{prop-convergent}. The coefficient $c(a,b)$ in \eqref{coeff-c1} is always at most $\av{C}\leq[F:K]$, and it is equal to $1$ if the condition on the Frobenius is trivial.

For the case $F=K=\Q$, Theorem \ref{thm-divisible} provides a new version of \cite[Theorem 1]{pappa} and \cite[Lemma 1]{moree}.  The main challenge for the generalization of Pappalardi's method consists in proving a certain unconditional version  of the Chebotarev density theorem for cyclotomic-Kummer extensions of number fields, as mentioned above. In Section \ref{sect-CDT} we will argue that this is not difficult if the base field $K$ is normal over $\Q$. However, for the general case we need an improvement on the upper bound of a possible Landau--Siegel zero of the Dedekind zeta function of $K_{m,n}$, which is achieved in Proposition \ref{prop-zero-bound}. Notice that Section \ref{sect-CDT} is self-contained and also of independent interest.

Section \ref{sect-proof-divis} is devoted to the proof of Theorem \ref{thm-divisible}, whereas in Section \ref{sect-dens} we justify that the natural density $\varrho_{C,m}$ is a positive rational number, and for $C=\Gal(F/K)$ (e.g.\ if $F=K$)  we express it in terms of finite sums and products, see Theorem \ref{thm-dens-formula}. 
 
Sections \ref{sect-kfree} and \ref{sect-val} are devoted to proving applications of Theorem \ref{thm-divisible}. In Section \ref{sect-kfree} we apply Theorem \ref{thm-divisible} to prove the following result on the primes $\p$ of $K$ for which $\ord_\p(G)$ is $k$-free.

\begin{thm}\label{thm-kfree}
Let $F/K$ be a finite Galois extension of number fields, and let $C$ be a conjugacy-stable subset of $\Gal(F/K)$. Let $G$ be a finitely generated and  torsion-free subgroup of $K^\times$ of positive rank $r$.  Let $k\geq2$ be an integer, and consider the following set of primes of $K$:
\[ \mc N_k :=\set{ \p: \ord_\p(G)\text{ is $k$-free},\,\begin{pmatrix}\p \\ F/K  \end{pmatrix}\subseteq C } \]
(where we are tacitly discarding the primes $\p$ that ramify in $F$ or such that $v_\p(g)\neq0$ for some $g\in G$). Then we have
\begin{equation}\label{eq-kfree}
\pi_{\mc N_k}(x) = \frac{x}{\log x}\sum_{m\geq1}\mu(m)\varrho_{C,m^k} +
O\left( \frac{x}{(\log x)^{1+\frac{k-1}{(r+2)k+1}}}  \right),
\end{equation}
where $\varrho_{C,m^k}$ is as in \eqref{eq-div-dens}. The set $\mc N_k$ has natural density
\[ \beta_{C,k}:= \sum_{m\geq1}\,\sum_{n\mid m^\infty}\,\sum_{d\mid m}\frac{\mu(m)\mu(d)c(nm^k,dn)}{[F_{nm^k,dn}:K]}, \]
where $c(a,b)$ is as in \eqref{coeff-c1}. The constant implied by the $O$-term depends only on $F$, $K$, $G$.
\end{thm}
Notice that the convergence of the series $\beta_{C,k}$ follows from Proposition \ref{prop-convergent}. 
For the case $F=K=\Q$ we refer to Pappalardi \cite[Theorem 1.2]{pappa-sf} and \cite[Remark (8), p.388]{pappa}.
In Section \ref{sect-val}, from Theorem \ref{thm-divisible} we also derive Theorem \ref{thm-val}, which concerns the  primes $\p$ for which the $\ell$-adic valuation of $\ord_\p(G)$ has a prescribed value for finitely many prime numbers $\ell$.

In Section \ref{sect-GRH}, under GRH, we derive some improvements on the error terms of formulas \eqref{eq-divisible} and \eqref{eq-kfree}. Finally, in Section \ref{sect-ex} we provide several numerical examples for the densities considered in this paper. 

\section{Chebotarev's density theorem for cyclotomic-Kummer extensions}\label{sect-CDT}
In this section we prove an effective  version of the Chebotarev density theorem for cyclotomic-Kummer extensions of number fields which is ``unconditional", i.e.\ it does not rely on GRH.  Let us first introduce some notation (in addition to the notation of Section \ref{not1}).
 
\subsection{Notation}\label{not2}
Given a finite Galois extension $L/K$ of number fields and a conju-gacy-stable subset $C$ of $\Gal(L/K)$, we denote by $\pi_{L/K,C}(x)$ the number of primes $\p$ of $K$ with $\N\p\leq x$ which are unramified in $L$ and such that $(\p,L/K)\subseteq C$. 

Also, $d_K$ denotes the absolute discriminant of $K$, $\OO_K$ the ring of integers of $K$, and $\zeta_K$ the Dedekind zeta function of $K$.
For a finitely generated subgroup $G$ of $K^\times$,  $\mc P(G)$ is the set of primes $\p$ in $K$ such that $v_\p(g)\neq0$ for some $g\in G$ (notice that this set is finite). 

Also, we use the following standard notation: $\varphi$ is Euler's totient function; given $m\geq1$, $\tau(m)$ is the number of positive divisors of $m$, and $\rad(m)$ is the radical of $m$, i.e.\ the largest squarefree integer dividing $m$;  $\Li(x)=\int_2^x\frac{dx}{\log x}$ is the logarithmic integral function.

\subsection{Chebotarev's density theorem}

We start by stating a recent effective and unconditional version of Chebotarev's density theorem by Thorner and Zaman \cite{TZ}. This result refines previous versions by Lagarias and Odlyzko \cite[Theorem 1.2]{lagarias}, Serre \cite[Theorem 2]{serre}, and V.K. Murty \cite[Section 4]{murty}.

\begin{thm}[{\cite[Theorem 1.1]{TZ}}]\label{thm-CDT-TZ}
Let $L/K$ be a finite Galois extension of number fields with $L\neq\Q$, and let $C$ be a conjugacy-stable subset of $\Gal(L/K)$. There exist absolute and effective constants $c_1,c_2>0$ such that, if
\begin{equation}\label{eq-condition1-TZ}
x\geq \big(\av{d_L}\cdot[L:\Q]^{[L:\Q]}\big)^{c_1},
\end{equation}
then $\pi_{L/K,C}(x)$ is given by
\begin{equation}\label{eq_TZ}
\frac{\av{C}}{[L:K]}\big(\Li(x)-\theta\Li(x^\beta)\big)\left(1 +\bo{e^{-\sqrt{\frac{c_2\log x}{[L:\Q]}}} +  e^{ - \frac{c_2\log x}{\log(\av{d_L}[L:\Q]^{[L:\Q]})} } }\right)
\end{equation}
where $\theta=\theta(C)\in\{-1,0,1\}$, and $\beta$ is a possible Landau--Siegel zero of the Dedekind zeta function $\zeta_L(s)$, i.e.\ the unique real zero in the range 
\[ 1-\frac{1}{4\log |d_L|}\leq\alpha\leq1 \]
(if it exists, otherwise $\theta=0$).
\end{thm}

In the formula that we aim to derive, the term $\frac{|C|}{[L:K]}\Li(x)$ should be dominant.
Thus the difficulty in applying Theorem \ref{thm-CDT-TZ} to cyclotomic-Kummer extensions consists in estimating the term $\Li(x^\beta)$, and hence bounding the value of $\beta$. As of today, the best known bound on $\beta$ is provided by Stark in \cite[Proof of Theorem 1', p.148]{stark}. In fact, if $K/\Q$ is normal, then that bound is good enough for our purpose, see Remark \ref{rem-normal}. However, for the general case we need to deduce from Stark's results some  improvement which is suitable for our goal.

\subsection{On a possible Landau--Siegel zero of the Dedekind zeta function}
\begin{lem}[{\cite[Lemma 3]{stark}}]\label{lem-stark1} 
Let $L\neq\Q$ be a number field. The Dedekind zeta function $\zeta_L(s)$ has at most one zero in the region
\begin{equation}\label{eqregion}
\set{s\in\mathbb C : 1-\frac{1}{4\log \av{d_L}}\leq\re(s)<1 \text{ and } \av{\im(s)}\leq \frac{1}{4\log \av{d_L}}}.
\end{equation}
If such a zero exists, then it is real and simple.
\end{lem}

\begin{lem}[{\cite[Lemma 8]{stark}}]\label{lem_8stark}
Let $L\neq\Q$ be a number field, and set $c_0(L)=4$ if $L/\Q$ is normal, $c_0(L)=4 \cdot[L:\Q]!$ otherwise. If $\zeta_L$ has a real zero $\beta$ such that $1-(c_0(L) \cdot \log|d_L|)^{-1}\leq\beta<1$, then there is a quadratic number field $M$ inside $L$ such that $\zeta_M(\beta)=0$.
\end{lem}

\begin{lem}\label{lem-stark2}
Let $L/K$ be a normal extension of number fields with $L\neq\Q$. If $\zeta_L$ has a real zero $\beta$ such that
\begin{equation}\label{ineq-beta}
1-\frac{1}{4(2[K:\Q])!\cdot\log\av{d_L}}\leq \beta< 1,
\end{equation}
then there is a quadratic number field $M$ inside $L$ such that $\zeta_M(\beta)=0$.
\end{lem}
The Lemma says, in particular, that if $L$ has no quadratic subfields, then $\zeta_L(s)$ has no real zero in the range \eqref{ineq-beta}.

\begin{proof}
If $K=\Q$, then the statement holds by Lemma \ref{lem_8stark}, hence we suppose $K\neq\Q$. If $\zeta_L(s)$ has a zero with real part lying in the range \eqref{ineq-beta}, then by Lemma \ref{lem-stark1} it must be real, simple, and unique in that range. Since $L/K$ is normal, by \cite[Theorem 3]{stark} there is a subextension $F$ of $L/K$, which is either trivial or quadratic over $K$, such that $\zeta_E(\beta)=0$ for every field $F\subseteq E\subseteq L$. Therefore we have that $\zeta_K(\beta)=0$ or $\zeta_F(\beta)=0$ for some quadratic extension $F/K$ with $F\subseteq L$. We conclude by applying Lemma \ref{lem_8stark} either to $K$ or $F$, noticing that the range of the lemma contains the interval \eqref{ineq-beta}. 
\end{proof}

\begin{prop}\label{prop-zero-bound}
Let $L/K$ be a Galois extension of number fields with $L\neq\Q$. Then the possible unique zero $\beta$ of the Dedekind zeta function $\zeta_L(s)$ in the region \eqref{eqregion} is real and simple, and we have
\begin{equation}\label{ineq-beta2}
\frac{1}{2}\leq\beta\leq \max\left\{ 1-\frac{1}{4(2[K:\Q])!\log \av{d_L}}, 1-\frac{1}{c_3\av{d_L}^{1/[L:\Q]}} \right\}  ,
\end{equation} 
where $c_3>0$ is an effective absolute constant.
\end{prop}

\begin{proof}
Clearly $\beta\geq1/2$ as $4\log \av{d_L}\geq2$. It suffices to show that if $\beta $ is in the range \eqref{ineq-beta}, then $\beta$ satisfies \eqref{ineq-beta2}. We follow the same argument as in \cite[Proof of Theorem 1', p.148]{stark}. If $L$ has no quadratic subfields, then, as  mentioned above,  $\zeta_L(s)$ has no real zero in the range \eqref{ineq-beta} by Lemma \ref{lem-stark2} and hence \eqref{ineq-beta2} is satisfied. If $L$ contains a quadratic field and $\zeta_L(\beta)=0$ for some $\beta$ in the range \eqref{ineq-beta}, then by Lemma \ref{lem-stark2} there must be a quadratic subfield $M$ of $L$ such that $\zeta_M(\beta)=0$. By \cite[Lemma 11]{stark} we must have $\beta<1-(c_3\av{d_M}^{1/2})^{-1}$, for an effective absolute constant $c_3>0$. We may conclude because we have $\av{d_L}\geq \av{d_M}^{[L:\Q]/2}$.  
\end{proof}

\begin{rem}\label{rem-normal}
In fact, if in the proof of Lemma \ref{lem-stark2} we have $\zeta_K(\beta)=0$, then the lower bound of \eqref{ineq-beta} may be taken as $1-(4[K:\Q]!\log\av{d_L})^{-1}$. Moreover, if $L$ is normal over $\Q$ or if there is a tower of normal extensions $\Q=k_0\subset k_1\subset\cdots\subset k_m=K$, then the lower bound of \eqref{ineq-beta} may be taken to be $1-(4\log\av{d_L})^{-1}$ by Lemma \ref{lem_8stark}, or $1-(16\log\av{d_L})^{-1}$ by \cite[Lemma 10]{stark}, respectively. Accordingly, the factor $4(2[K:\Q])!$ in \eqref{ineq-beta2} can be replaced by $4$ or $16$ in the respective cases, by following the same argument of the proof of Proposition \ref{prop-zero-bound}.
\end{rem}

\subsection{Chebotarev's density theorem for cyclotomic-Kummer extensions}

\begin{prop}\label{prop-disc}
Let $K$ be a number field, and let $G$ be a finitely generated and  torsion-free subgroup of $K^\times$ of positive rank $r$. Then, for $m,n\geq1$ with $n\mid m$, we have
\[ \frac{\log \av{d_{K_{m,n}}}}{[K_{m,n}:\Q]}\leq \log(\varphi(m)mn^r) + \log \big(\sigma_G\av{d_K}\big), \]
where $\sigma_G$ is the product of the rational primes lying below the primes in $\mc P(G)$.
\end{prop}

\begin{proof}
Given a finite extension $L/K$ we write $d_{L/K}$ for the relative discriminant. We have
\begin{equation}\label{eq-disc1}
d_{K_{m,n/\Q}}=\N_{K/\Q}(d_{K_{m,n}/K})\cdot d_{K/\Q}^{[K_{m,n}:K]}\,,
\end{equation}
see \cite[Ch.III, Corollary 2.10]{neukirch}. By \cite[Proposition 5]{serre}, since $K_{m,n}/K$ is Galois,  we have
\begin{equation}\label{eq-disc2}
\log\av{\N_{K/\Q}(d_{K_{m,n}/K})}\leq [K_{m,n}:\Q]\bigg(\log[K_{m,n}:K]+\sum_{p\in P(K_{m,n}/K)}\log p \bigg)
\end{equation}
where $P(K_{m,n}/K)$ is the set of rational primes $p$ lying below the primes of $K$ that ramify in $K_{m,n}$. These prime numbers divide $m$ or $\sigma_G$, as they lie below the primes $\p$ that divide $d_{K_{m,n}/K}$, and an estimate for this relative discriminant is \cite[Formula (4.7)]{PS1}:
\[ d_{K_{m,n}/K}\mid \Big(mn^r\prod_{i=1}^r(\alpha_i\beta_i)^2\Big)^{n^r\varphi(m)}\OO_K\,, \]
where $\alpha_i,\beta_i\in\OO_K$ are such that the elements $\gamma_i:=\alpha_i/\beta_i$ for $i\in\{1,\ldots,r\}$ form a basis of $G$ as a free $\Z$-module. Since $n\mid m$, we have
\begin{equation}\label{eq-disc3}
\sum_{p\in P(K_{m,n}/K)}\log p\leq \log (m\sigma_G)\,.
\end{equation}
We conclude by taking the natural logarithm of $\av{d_{K_{m,n}}}$, making use of  \eqref{eq-disc1}, and by applying the bounds \eqref{eq-disc2}, \eqref{eq-disc3} and $[K_{m,n}:K]\leq\varphi(m)n^r$.
\end{proof}

We are now ready to prove an effective unconditional Chebotarev density theorem for cyclotomic-Kummer extensions of number fields.

\begin{thm}\label{thm-chebTZ-kummer} 
Let $F/K$ be a Galois extension of number fields, and let $G$ be a finitely generated and  torsion-free subgroup of $K^\times$ of positive rank $r$. Let $C$ be a conjugacy-stable subset of $\Gal(F/K)$, and for all integers $m,n\geq1$ with $n\mid m$, define
\begin{equation}\label{eq-Cmn}
C_{m,n}:=\{\sigma\in \Gal(F_{m,n}/K) : \sigma|_F\in C,\, \sigma|_{K_{m,n}}=\id  \}\,,
\end{equation}
which is a conjugacy-stable subset of $\Gal(F_{m,n}/K)$. Then there exist constants $c_4,c_5>0$, which depend only on $F$ and $G$, such that, for all  $b>0$ and uniformly for
\begin{equation}\label{eq-condition2}
m\leq c_4\left( \frac{\log x}{(\log\log x)^{b}} \right)^{\frac{1}{r+2}}\,,
\end{equation}
we have
\begin{equation}\label{eq_chebTZ_kummer}
\pi_{F_{m,n}/K,C_{m,n}} (x)=
\frac{\av{C_{m,n}}}{[F_{m,n}:K]}\Li(x)+
O_{F,G}\left(\frac{x}{\log x\cdot e^{c_5(\log\log x)^{b}}}\right).
\end{equation}
\end{thm}

We first prove an upper bound for $\pi_{F_{m,n}/K,C_{m,n}} (x)$.

\begin{prop}\label{prop-BT}
With the notation of Theorem \ref{thm-chebTZ-kummer}, there exists a constant $c_6>0$ (depending only on $F$ and $G$), such that, for all $b>0$ and uniformly for
\begin{equation}\label{eq-condition3}
m\leq c_6\left(\frac{\log x}{(\log\log x)^{1+b}}\right)^{\frac{1}{r+1}},
\end{equation}
we have
\[ \pi_{F_{m,n}/K,C_{m,n}}(x)\leq (2+o(1))\frac{\av{C_{m,n}}}{[F_{m,n}:K]}\Li(x). \]
\end{prop}

\begin{proof}
We apply Theorem \ref{thm-CDT-TZ} to $F_{m,n}/K$ and $C_{m,n}$. By Proposition \ref{prop-disc}  and since $[F_{m,n}:F]\leq m^{r+1}$ we have
\begin{multline*}
\log\av{d_{F_{m,n}}} + [F_{m,n}:\Q]\log[F_{m,n}:\Q] \\
\leq [F:\Q]m^{r+1}\big(\log \big([F:\Q]m^{r+2}\big)+\log \big(\sigma_G\av{d_F}\big)  \big)  \leq c_{7} \cdot m^{r+1}\log m,
\end{multline*}
where $\sigma_G$ is as in Proposition \ref{prop-disc}. 
Thus, for \eqref{eq-condition1-TZ} to be satisfied, we should have at most $m\ll ( \frac{\log x}{\log\log x} )^{\frac{1}{r+1}}$. However, for the second error term in \eqref{eq_TZ} to be $o(1)$, we require the bound \eqref{eq-condition3} on $m$, which results in the estimate
 \[ m^{r+1}\log m\ll_{F,G} \frac{\log x}{(\log\log x)^b}. \]
Notice that the constant $c_6$ is chosen in such a way that $c_1c_{7}\cdot m^{r+1}\log m\leq \log x$.
 As for the  first error term of  \eqref{eq_TZ}, we have 
\[ \exp\left(-\sqrt{\frac{c_2\log x}{[F_{m,n}:\Q]}}\right)  \leq  \exp\Big(-c_8(\log\log x)^{\frac{1+b}{2}}\Big)=o(1)\,.  \]
Applying the inequality $0<\Li(x)-\theta\Li(x^\beta)< 2\Li(x)$ yields the result. The constants $c_1,c_2$ are as in Theorem \ref{thm-CDT-TZ}, and $c_7,c_8>0$  depend only on $F$ and $G$.
\end{proof}

\begin{proof}[Proof of Theorem \ref{thm-chebTZ-kummer}]
We apply Theorem \ref{thm-CDT-TZ} to $F_{m,n}/K$ and $C_{m,n}$. Since the bound \eqref{eq-condition2} is smaller than \eqref{eq-condition3}, the condition \eqref{eq-condition1-TZ} is satisfied by the proof of Proposition \ref{prop-BT}. The constant $c_4$ is chosen similarly as for $c_6$. The two error terms in \eqref{eq_TZ} are bounded by
\[ \exp\left(-\frac{c_2\log x}{c_7\,m^{r+1}\log m}\right) \leq
\exp\left( -\frac{c_9(\log x)^{\frac{1}{r+2}}}{\log\log x} \right) \]
and 
\[ \exp\left(-\sqrt{\frac{c_2\log x}{[F_{m,n}:\Q]}}\right)  \leq  \exp\Big(-c_{10}(\log x)^{\frac{1}{2(r+2)}}\cdot(\log\log x)^{\frac{b}{4}}\Big)\,.  \]
Applying the inequality $0<\Li(x)-\theta\Li(x^\beta)<2\Li(x)$, both $O$-terms are smaller than the one in the statement.
We now bound the term $\Li(x^\beta)$. Since $\beta\geq1/2$, and denoting $\alpha:=(1-\beta)^{-1}$, we have
\[ \Li(x^\beta)\ll\frac{x^\beta}{\log x} = \frac{x}{\log x}\exp\bigg(-\frac{\log x}{\alpha}\bigg), \]
hence we need to bound $\alpha$ with a function $h(x)$ satisfying $\exp(-\frac{\log x}{h(x)})=o(1)$.
Taking into account the two terms in the upper bound on $\beta$ from Proposition \ref{prop-zero-bound}, on the one hand, we have
\[ \alpha\leq 4(2[F:\Q])!\log \av{d_{F_{m,n}}} \ll_{F,G} m^{r+1}\log m, \]
so $x^\beta/\log x$ is bounded as above.
On the other hand, we have
\begin{align*}
\alpha\leq c_3\av{d_{F_{m,n}}}^{1/[F_{m,n}:\Q]} & \leq c_3\exp\left(\log m^{r+2}+ \log \big(\sigma_G\av{d_F}\big)\right) \\
& \ll_{F,G} m^{r+2}\ll_{F,G} \frac{\log x}{(\log\log x)^{b}},
\end{align*}
which satisfies our requirement and yields the bound
\[ \frac{x^\beta}{\log x} \leq \frac{x}{\log x\cdot e^{c_5(\log\log x)^{b}}}. \]
Collecting all error terms gives the asymptotic formula in the statement. The constants $c_5,c_9,c_{10}>0$ depend only on $F$ and $G$.
\end{proof}

\section{The  order being divisible by a given integer}\label{sect-proof-divis}

In this section we prove Theorem \ref{thm-divisible}. We first set some notation, in addition to the notation of Sections \ref{not1} and \ref{not2}.

\subsection{Notation}\label{not3}
Let $F/K$ be a Galois extension of number fields, $C$ a conjugacy-stable subset of $\Gal(F/K)$, and $G$ a finitely generated and  torsion-free subgroup of $K^\times$.
We say that  $\p$ is a prime of \emph{degree} $1$ in $K$ if its ramification index and residue class degree over $\Q$ are equal to $1$. 
For $m,n\geq1$ with $n\mid m$ we define $\pi^1_{m,n}(x)$ to be the number of primes $\p$ of $K$, with $\N\p\leq x$, which are of degree $1$, split completely in $K_{m,n}$, do not ramify in $F$, and satisfy $(\p, F/K)\subseteq C$.  In other words,  $\pi^1_{m,n}(x)$ is the number of primes counted by $\pi_{F_{m,n}/K,C_{m,n}}(x)$ which are of degree $1$, where $C_{m,n}$ is as in \eqref{eq-Cmn} (we are fixing $K$, $F$, $G$, and $C$). 

For $\p\notin \mc P(G)$, recall that $\ord_\p(G)$ is the order of $(G\bmod\p)$, and we also denote by $\ind_\p(G)$ the index of $(G\bmod\p)$, namely
\[ \ind_\p(G)=[k_\p^\times:\group{G\bmod\p}]=(\N\p-1)/\ord_\p(G). \]

Given integers $m,n\geq1$, we write $(m,n)$ for $\gcd(m,n)$. 
If $m$ is replaced by the supernatural number $m^\infty$,  we write $(n,m^\infty)=\prod_{\ell\mid m\text{ prime}}\ell^{v_\ell(n)}$, where $v_\ell$ is the $\ell$-adic valuation.

\subsection{Proof of Theorem \ref{thm-divisible}}
Recall that if $\p$ is a prime of $K$ of degree $1$ such that $\p\notin \mc P(G)$, then $\N\p\equiv1\bmod n$ if and only if $\p$ splits completely in $K(\zeta_n)$, and $n\mid\ind_\p(G)$ if and only if $\p$ splits completely in $K_{n,n}$, where $n\geq1$; see also \cite[Lemma 2]{zieg}. Hence, we easily deduce that, for $m,n\geq1$ with $n\mid m$, we have
\begin{equation}\label{eqcriterion}
\pi^1_{m,n}(x)=\cset{\p: \N\p\leq x,\,\N\p\equiv1\bmod m,\, n\mid\ind_\p(G),\, \begin{pmatrix}\p \\ F/K  \end{pmatrix}\subseteq C },
\end{equation}
(where we only consider primes of degree $1$). Notice that, since the number of primes of $K$ not of degree $1$ is of order $O(\sum_{p\leq\sqrt{x}}1)=\bo{\sqrt{x}/\log x}$, see e.g.\ \cite[Lemma 1]{zieg}, the same asymptotic formula \eqref{eq_chebTZ_kummer} gives $\pi_{m,n}^1(x)$. 

\begin{lem}\label{lem-key-id}
If $\mc P_m$ is as in Theorem \ref{thm-divisible}, then we have
\[ \pi_{\mc P_m}(x)=\sum_{n\mid m^\infty}\sum_{d\mid m}\mu(d)\pi^1_{mn,dn}(x)+\bo{\frac{\sqrt{x}}{\log x}}. \]
\end{lem}

\begin{proof}
The proof is a variation of \cite[Proof of Proposition 1]{moree}. The $O$-term estimates the primes of $K$ which are not of degree $1$.
Let $\p\in\mc P_m$ be a prime of degree $1$, and let $\N\p=p$. Then we have $m\mid (p-1)$ and there is a unique $n\mid m^\infty$ such that $p\equiv1\bmod mn$, $n\mid\ind_\p(G)$ and $(\frac{\ind_\p(G)}{n},m)=1$ (we must have $n=(\ind_\p(G),m^\infty)$). Hence we can write
\[ \pi_{\mc P_m}(x)=\sum_{n\mid m^\infty}\pi_{\mc B_n}(x)+\bo{\frac{\sqrt{x}}{\log x}}\,, \]
where for $n\mid m^\infty$ we set
\begin{equation}\label{def_Bn}
\mc B_n := \set{\p:  p\equiv1\bmod mn,\,n\mid\ind_\p(G),\, \Big(\frac{\ind_\p(G)}{n},m\Big)=1,\, \begin{pmatrix}\p \\ F/K  \end{pmatrix}\subseteq C }
\end{equation}
(we are tacitly assuming that the primes in $\mc B_n$ are of degree $1$, do not lie in $\mc P(G)$ and do not ramify in $F$). Notice that, $\p\in\mc B_n$ satisfies $m\mid \ord_\p(G)$ because of the two conditions $p\equiv1\bmod mn$ and $\big(\frac{\ind_\p(G)}{n},m\big)=1$ and the identity $\ord_\p(G){\cdot}\ind_\p(G)=p-1$.

Next we apply the inclusion-exclusion principle to the condition  $\big(\frac{\ind_\p(G)}{n},m\big)=1$, which amounts to $n\mid\ind_\p(G)$ and $n\ell\nmid \ind_\p(G)$ for all primes $\ell\mid m$, and so  we obtain
\[ \pi_{\mc B_n}(x) =\sum_{d\mid m}\mu(d)\cset{\p:p\leq x,\, p\equiv1\bmod mn,\,dn\mid\ind_\p(G),\, \begin{pmatrix}\p \\ F/K  \end{pmatrix}\subseteq C }\,. \]
We conclude by \eqref{eqcriterion} that
\begin{equation}\label{eq-B}
\pi_{\mc B_n}(x) =\sum_{d\mid m}\mu(d)\pi^1_{mn,dn}(x)\,.\qedhere
\end{equation}
\end{proof}

\begin{rem}\label{rem-bounds}
It follows from the definition of $\mc B_n$ \eqref{def_Bn} that
\[ \pi_{\mc B_n}(x)\leq [K:\Q]\cdot\cset{p\leq x:p\equiv1\bmod mn}, \]
and $\pi_{\mc B_n}(x)\leq \pi_{mn,n}^1(x)$. From this last inequality, identity \eqref{eq-B}, and by the Chebotarev density theorem we deduce
\[ 0\leq \sum_{d\mid m}\frac{\mu(d)\av{C_{mn,dn}}}{[F_{mn,dn}:K]}\leq \frac{\av{C_{mn,n}}}{[F_{mn,n}:K]}\leq \frac{1}{[K_{mn,n}:K]}. \] 
\end{rem}

We are now ready to prove Theorem \ref{thm-divisible}.

\begin{proof}[Proof of Theorem \ref{thm-divisible}]
Notice that, for $m,n\geq1$ with $n\mid m$, the coefficient  $c(m,n)=| C\cap\Gal(F/F\cap K_{m,n}) |$ defined in \eqref{coeff-c1} is equal to $|C_{m,n}|$, because if $\sigma\in C$ fixes $F\cap K_{m,n}$, then $\sigma$ can be lifted to a unique element of $C_{m,n}$. We are going to apply Lemma \ref{lem-key-id} and Theorem \ref{thm-chebTZ-kummer} with $b=2$ (in fact, $b>1$ suffices). For $n\mid m^\infty$ let $\mc B_n$ be as in \eqref{def_Bn}, and recall \eqref{eq-B}.
Set $y:=c_4( \log x/(\log\log x)^2 )^{1/(r+2)}$, where $c_4>0$ is the constant of Theorem \ref{thm-chebTZ-kummer}. Thus, we have 
\begin{align*}
\pi_{\mc P_m}(x)=&\sum_{\substack{n\mid m^\infty\\ nm\leq y}}\sum_{d\mid m}\mu(d)\pi^1_{mn,dn}(x)+O\Bigg(\sum_{\substack{n\mid m^\infty\\ nm> y}}\pi_{\mc B_n}(x)\Bigg)+\bo{\frac{\sqrt{x}}{\log x}} \\
=&\Li(x)\sum_{\substack{n\mid m^\infty\\ nm\leq y}}\sum_{d\mid m}\frac{\mu(d)c(mn,dn)}{[F_{mn,dn}:K]}+
 O_{F,G}\Bigg(\frac{\tau(m)}{m}\frac{x\cdot y}{\log x\cdot e^{c_5(\log\log x)^2}}\Bigg) \\
& +O\Bigg(\sum_{\substack{n\mid m^\infty\\ nm> y}}\pi_{\mc B_n}(x)\Bigg)+\bo{\frac{\sqrt{x}}{\log x}}.
\end{align*}
In order to estimate the tail of the series in the main term we make use of Remark \ref{rem-bounds} and obtain
\begin{align}
\pi_{\mc P_m}(x)=&\Li(x)\sum_{n\mid m^\infty}\sum_{d\mid m}\frac{\mu(d)c(mn,dn)}{[F_{mn,dn}:K]}
+O\Bigg(\frac{x}{\log x}\sum_{\substack{n\mid m^\infty \\ mn>y}}\frac{1}{[K_{mn,n}:K]}\Bigg) \label{eq-error1}\\
& +O_{F,G}\Bigg(\frac{x\cdot y}{\log x\cdot e^{c_5(\log\log x)^2}}\Bigg)+O\Bigg(\sum_{\substack{n\mid m^\infty\\ nm> y}}\pi_{\mc B_n}(x)\Bigg). \label{eq-error2}
\end{align}

The first error term in \eqref{eq-error2} is negligible with respect to the error term in the statement. Let us estimate the error term in \eqref{eq-error1}. Since $[K_{mn,n}:K]\gg_{K,G} \varphi(mn)n$ by \cite[Theorem 3.1]{PS1} and $\varphi(mn)=\varphi(m)n$ (as $\rad(n)\mid m$), we can bound
\[
\sum_{\substack{n\mid m^\infty \\ mn>y}}\frac{1}{[K_{mn,n}:K]}\ll_{K,G} 
\frac{1}{\varphi(m)}\sum_{ n> y/m}\frac{1}{n^2} 
\ll_{K,G}  \frac{m}{\varphi(m)y}.
\]
As $m/\varphi(m)=\bo{\log\log m}$, see e.g.\ \cite[Theorem 15]{ross}, and $m\leq x$ without loss of generality, we then have
\begin{equation}\label{p-error1}
\frac{x}{\log x}\sum_{\substack{n\mid m^\infty \\ mn>y}}\frac{1}{[K_{mn,n}:K]}\ll_{K,G}\frac{x\log\log x}{y\log x}.
\end{equation}

Next we focus on the second error term in \eqref{eq-error2}. First notice that, since $\frac{mn}{\varphi(mn)}=\frac{m}{\varphi(m)}$, by \cite[Lemma 3.3]{pappa-sf} for $T\geq1$ and  $0<c<1$,  we have
\begin{equation}\label{eqlemma-pappa}
\sum_{\substack{n\mid m^\infty\\ nm> T}}\frac{1}{\varphi(mn)} = \frac{m}{\varphi(m)}\sum_{\substack{n\mid m^\infty\\ nm> T}}\frac{1}{mn} \ll_{c} \frac{m}{\varphi(m)}\frac{1}{T^{c}}.
\end{equation}
We bound $\pi_{\mc B_n}(x)$ first with the number of primes $\p$ with $\N\p\leq x$ splitting completely in $K(\zeta_{mn},\alpha^{1/n})$ for some $\alpha\in G$, $\alpha\neq1$,  and then with $[K:\Q]$ times the number of multiples of $mn$ up to $x$ (see Remark \ref{rem-bounds}).
Hence, setting $z:=(\log x)^{2/c}$, we have
\begin{equation}\label{eq_BTprep}
\sum_{\substack{n\mid m^\infty\\ nm> y}}\pi_{\mc B_n}(x) 
\ll_K\sum_{\substack{n\mid m^\infty \\ y< nm\leq z}}\pi_{K(\zeta_{mn},\alpha^{1/n})/K,\id}(x)+
\sum_{\substack{n\mid m^\infty \\ nm>z}}\cset{k\leq x: mn\mid k}
\end{equation}
Applying \eqref{eqlemma-pappa}, the second term in this last expression is bounded by 
\[ \sum_{\substack{n\mid m^\infty \\ nm>z}}\frac{x}{mn}\ll \frac{x}{\log^2x} .\]
Setting $w:=c_6\frac{\sqrt{\log x}}{\log\log x}$, where $c_6>0$ is as in \eqref{eq-condition3}, we split the  first sum in the right-hand side of \eqref{eq_BTprep} into two parts.
Thanks to Proposition \ref{prop-BT} (applied with $r=b=1$) we bound the sum with range $y<mn\leq w$ by
\[ \frac{x}{\log x}\sum_{\substack{n\mid m^\infty \\ y< nm\leq w}}\frac{1}{[K(\zeta_{mn},\alpha^{1/n}):K]} 
\ll \frac{x}{\varphi(m)\log x}\sum_{n>y/m}\frac{1}{n^2},  \]
and this is estimated as in \eqref{p-error1}. Invoking Remark \ref{rem-bounds}, and applying  the Brun--Titchmarsh Theorem, we  estimate the  sum with range $w<mn\leq z$ by
\[  \sum_{\substack{n\mid m^\infty \\ w< nm\leq z}}\frac{x}{\varphi(mn)\log(x/mn)} 
\leq \frac{x}{\log(x/z)}\sum_{\substack{n\mid m^{\infty} \\ mn>w}} \frac{1}{\varphi(mn)} \ll_{K,G}
\frac{x(\log\log x)^{9/5}}{(\log x)^{7/5}},  \]
where we used \eqref{eqlemma-pappa} with $c=4/5$. This term is bounded by the error term in the statement. 

Notice that we may replace $\Li(x)$ with $x/\log x$ in the main term of \eqref{eq-error1} because $\Li(x)=x/\log x+O(x/\log^2 x)$ and the series by which it is multiplied is convergent (see Proposition \ref{prop-convergent}).
\end{proof}

\subsection{Properties and remarks}

\begin{rem}\label{rem-pappa}
One can see from \cite[Proof of Lemma 3.3]{pappa-sf} that the constant depending on $c$ arising in \eqref{eqlemma-pappa} can be taken equal to
\[ \prod_{\substack{p  \text{ prime} \\ p^{1-c}\leq 2}}\frac{1}{p^{1-c}-1}.  \]
Thus, in \cite[Proof of Theorem 1]{pappa} the parameter $0<c<1$ should be independent of $x$, and this would result in an error term having the exponent $1+\frac{c}{3(r+1)}$. Although our technique avoids this problem, it is still useful to see how to deal with $c$ depending on $x$.
Let $m\geq1$, and $\omega(m)$ the number of prime factors of $m$. 
By the Mean value theorem we have $1-1/p^b>bp^{-b}\log p$, with $0<b<1$ and $p$ any prime number. Thus, from \cite[Proof of Lemma 3.3]{pappa-sf} we obtain 
\[ \sum_{\substack{k> T \\ m\mid k\mid m^\infty}}\frac{1}{k} \leq 
\frac{1}{m^{1-c}T^c}
\prod_{\substack{p\mid m \\ \text{prime}}}\left(1-\frac{1}{p^{1-c}}\right)^{-1}
\leq\frac{2(1-c)^{-\omega(m)}}{T^c}.\]
Taking $c=1-1/\log\log x$ and making use of this inequality in \cite[Proof of Theorem 1]{pappa} reduces the final error term to 
\[  x(\log\log x)^{\omega(m)}\bigg( \frac{(\log\log x)^2}{\log x} \bigg)^{1+\frac{1}{3(r+1)}} .
 \]
\end{rem}

\begin{prop}\label{prop-convergent}
The series $\varrho_{C,m}$ from Theorem \ref{thm-divisible} is convergent, and we have
\[ \varrho_{C,m}\ll_{K} \frac{\zeta(r+1)}{\varphi(m)} , \]
where $\zeta$ is the Riemann zeta function.
\end{prop}

\begin{proof}
Applying Remark \ref{rem-bounds} and the estimate $[K_{mn,n}:K]\gg_{K,G}\varphi(mn)n^r=\varphi(m)n^{r+1}$ (see \cite[Theorem 3.1]{PS1}), we have
\[ \varrho_{C,m} \leq \sum_{n\mid m^\infty}\frac{1}{[K_{mn,n}:K]} 
\ll_{K,G}\frac{1}{\varphi(m)}\sum_{n\mid m^\infty}\frac{1}{n^{r+1}}
\leq \frac{1}{\varphi(m)}\zeta(r+1). \qedhere \]
\end{proof}

\begin{rem}\label{rem-torsion}
The assumption that the group $G$ is torsion-free allows some simplifications throughout the proof of Theorem \ref{thm-divisible}. However, the general case can be treated easily. Let $G'$ be a finitely generated subgroup of $K^\times$ with torsion,  and write $G'=G\times\group{\zeta_t}$, where $\zeta_t\in K^\times$, $t\geq2$, and $G\subseteq K^\times$ is torsion-free. Then, for all primes $\p$ of $K$ of norm large enough, we have
\[ m\mid\ord_\p(G') \qquad \text{ if and only if } \qquad \prod_{\substack{\ell\mid m\text{ prime} \\ v_\ell(m)> v_\ell(t)}}\ell^{v_\ell(m)} \mid \ord_\p(G). \]
\end{rem}

\begin{rem}\label{rem-coeff}
Let us consider the expression of the density $\varrho_{C,m}$ for some special cases of $C$. If $C=\Gal(F/K)$, then the condition on the Frobenius becomes trivial and $c(a,b)=[F:F\cap K_{a,b}]$. Therefore, we obtain
\[
\varrho_m:=\varrho_{\Gal(F/K),m}=\sum_{n\mid m^\infty}\sum_{d\mid m}\frac{\mu(d)}{[K_{mn,dn}:K]}.
\]
If $C=\set{\id}$, then the condition $(\p,F/K)=\id$ is equivalent to $\p$ splitting completely in $F$. In this case $c(a,b)=1$, and hence $\varrho_{\set{\id},m}$ equals $1/[F:K]$ times the density of primes $\mathfrak{P}$ of $F$ such that $m\mid \ord_\mathfrak{P}(G)$.

Finally, if $F$ is linearly disjoint over $K$ from $K_{a,b}$, then $c(a,b)=\av{C}$. Hence, if this holds for all $a,b$ we obtain $\varrho_{C,m}=\varrho_m\cdot\av{C}/[F:K]$.

Clearly, analogous statements hold for the densities $\beta_{C,k}$ of Theorem \ref{thm-kfree} and $\gamma_{C,k,m}$ of Theorem \ref{thm-val}.
\end{rem}

\section{A rational formula for the density}\label{sect-dens}
As a special case of \cite[Corollary 7]{PS2}, the natural density $\varrho_{C,m}$ of the set $\mc P_m$ from Theorem \ref{thm-divisible} is a positive rational number. In this section   we also provide an explicit closed formula for $\varrho_{C,m}$ when the condition on the Frobenius is trivial (in this case we write $\varrho_m$ for $\varrho_{C,m}$, as in Remark \ref{rem-coeff}).
In the rest of the paper, $\ell$ will always represent a prime number (also when not mentioned explicitly). Notice that, over $\Q$, Pappalardi \cite{pappa} provided an explicit rational formula for $\varrho_m$ for $G$ consisting of positive rationals, whereas the case of groups with negative rationals was considered in Abdullah et al.\ \cite{andam}.

\begin{thm}\label{thm-dens-formula}
Let $K$ be a number field and let $G$ be a finitely generated and torsion-free subgroup of $K^\times$ of positive rank $r$.
Let $m\geq1$ be an integer and let $\varrho_{m}$ be the natural density of the set of primes $\p$ of $K$ such that $m\mid \ord_\p(G)$ (where $\p\notin\mc P(G)$).
Then there is an integer $z$, which depends only on $K$ and $G$, such that
\begin{equation}\label{euler-prod-1}
\varrho_{m}=\frac{1}{\varphi(m)}\prod_{\substack{\ell\mid m \\ \ell\nmid z}}\frac{\ell(\ell^r-1)}{\ell^{r+1}-1}\cdot  \sum_{\substack{g\mid z \\ \rad(g)\mid(m,z)\mid g}}\sum_{h\mid g}\frac{p(g,h)}{[K_{g,h}:K]},
\end{equation}
where we set $p(g,h)=0$ if and only if at least one of the following conditions holds: 
\begin{itemize}
\item there is $\ell\mid g$, $\ell\nmid h$ such that $v_\ell(g/(m,z))>0$,
\item there is $\ell\mid h$ such that $v_\ell(z/g)>0$ and $v_\ell(g/h)\notin\{ v_\ell(m),v_\ell(m)-1 \}$,
\item there is $\ell\mid h$ such that  $v_\ell(z/g)=0$ and $v_\ell(g/h)>v_\ell(m)$;
\end{itemize}
else we define
\[ p(g,h)=\frac{\varphi(g)}{h}\cdot
\prod_{\substack{\ell\mid h, \ v_\ell(\frac{z}{g})>0 \\ v_\ell(\frac{g}{h})=v_\ell(m)-1}} -\ell \cdot 
\prod_{\substack{\ell\mid h, \ v_\ell(\frac{z}{g})=0 \\ 1\leq v_\ell(\frac{g}{h})< v_\ell(m)}} -(\ell-1) \cdot 
\prod_{\substack{\ell\mid h \\ v_\ell(\frac{z}{h})=0}}\frac{-\ell^{r+1}(\ell-1)}{\ell^{r+1}-1}. \]
\end{thm}

Notice that the formula for $\varrho_{m}$ involves only finite sums and  products. Moreover, for a general number field $K$ and a finitely generated and torsion-free group $G\subseteq K^\times$, the integer $z$ is explicitly described by the results of \cite{PST} (see e.g.\ \cite[Theorem 1.2 and its proof]{PST}).  Also, notice that for all $m$ such that $(m,z)=1$ we have
\[ \varrho_{m}=\frac{\rad(m)}{\varphi(m)}\prod_{\ell\mid m}\frac{\ell^r-1}{\ell^{r+1}-1}. \]

The rational number characterized by $z$ in \eqref{euler-prod-1}, i.e.\ obtained by isolating the dependence on $z$ of $\varrho_m$,  should be interpreted as a correction factor due to the entanglement between the radicals of elements in $G$ and the roots of unity in cyclotomic extensions of $K$. The same remark applies to the case of $k$-free orders in Theorem \ref{cor-kfree}.
The correction factors arising in problems related to Artin's primitive root conjecture have been the object of study of several articles; see, for instance, \cite{JPS,LMS,Stev}.

\begin{proof}[Proof of Theorem \ref{thm-dens-formula}]
By \cite[Theorem 1.1]{PST} there is an integer $z$, which depends only on $K$ and $G$, such that for $n\mid m$ we have
\[ [K_{m,n}:K]=\frac{\varphi(m)n^r}{\varphi((m,z))(n,z)^r}\cdot [K_{(m,z),(n,z)}:K]\,. \]
Therefore, we have 
\begin{equation}\label{euler-prod-2}
\varrho_{m} =\sum_{n\mid m^\infty}\sum_{d\mid m}\frac{\mu(d)}{[K_{mn,dn}:K]}= \frac{1}{\varphi(m)}\sum_{\substack{g\mid z\\ h\mid g}} \frac{\varphi(g)h^r}{[K_{g,h}:K]}\sum_{\substack{n\mid m^\infty \\ (mn,z)=g}}\ \sum_{\substack{d\mid m \\ (dn,z)=h}}\frac{\mu(d)}{n^{r+1}d^r}.
\end{equation}

First of all, for all $n\mid m^\infty$ we have that $\rad(mn,z)=\rad(m,z)$, hence we may restrict the sum on $g\mid z$ to the divisors such that $(m,z)\mid g$ and $\rad(g)=\rad((m,z))$ both hold. To simplify the notation, let us write $m_\ell=v_\ell(m)$, and similarly for $z_\ell,g_\ell,h_\ell$. Then, by properties of the multiplicative functions, from \eqref{euler-prod-2} we obtain
\[ \varrho_{m}=\frac{1}{\varphi(m)}\cdot\sum_{\substack{g\mid z  \\ \rad(g)\mid(m,z)\mid g}} \sum_{h\mid g}\frac{\varphi(g)h^r}{[K_{g,h}:K]}\cdot
\prod_{\ell\mid m} p_\ell(g,h) \]
where for $\ell\mid m$ we define
\begin{equation}\label{eq-prod-sum}
 p_\ell(g,h):=  \sum_{\substack{s\geq 0 \\ \min(m_\ell+s,z_\ell)=g_\ell}}\ \sum_{\substack{e\in\{0,1\} \\ \min(s+e,z_\ell)=h_\ell}}  \frac{\mu(\ell^e)}{\ell^{s(r+1)}\ell^{er}}. 
\end{equation}
If $\ell\mid m$ and $\ell\nmid z$, then the two conditions on the indices are trivial and we have
\[ p_\ell(g,h)=\frac{\ell(\ell^r-1)}{\ell^{r+1}-1}. \]
This computation already justifies the first product in \eqref{euler-prod-1}. Next, we take 
\[ p(g,h):=\varphi(g)h^r \prod_{\ell\mid g}p_\ell(g,h), \]
and compute $p_\ell(g,h)$ depending on the prime factors $\ell$ of $g$ (equivalently, of $(m,z)$). 
 
\emph{Case 1: $\ell\mid g$ and $\ell\nmid h$.}  Since $\ell\nmid h$,  the conditions on the indices in \eqref{eq-prod-sum} hold only for $s=e=0$, and so $p_\ell(g,h)=1$ if $\min(m_\ell,z_\ell)=g_\ell$, and $p_\ell(g,h)=0$ otherwise.

\emph{Case 2: $\ell\mid h$ and $g_\ell<z_\ell$.} 
Since $1\leq h_\ell<z_\ell$, the conditions on the indices hold only for $s+e=h_\ell$. Therefore, if $g_\ell=m_\ell+h_\ell$, then $p_\ell(g,h)=1/\ell^{h_\ell(r+1)}$; if $g_\ell=m_\ell+h_\ell-1$, then $p_\ell(g,h)=-\ell/\ell^{h_\ell(r+1)}$; otherwise $p_\ell(g,h)=0$.

\emph{Case 3: $\ell\mid h$ and $h_\ell< g_\ell=z_\ell$.} 
The conditions on the indices hold only for $s+e=h_\ell$ and $m_\ell+s\geq z_\ell=g_\ell$. Therefore, if $m_\ell+h_\ell-1\geq z_\ell$, then $p_\ell(g,h)=-(\ell-1)/\ell^{h_\ell(r+1)}$; if $m_\ell+h_\ell= z_\ell$, then $p_\ell(g,h)=1/\ell^{h_\ell(r+1)}$; otherwise $p_\ell(g,h)=0$.

\emph{Case 4: $\ell\mid h$ and $h_\ell=z_\ell$.} 
Since $h_\ell=g_\ell=z_\ell\geq1$, the conditions on the indices hold if and only if $s+e\geq h_\ell$. Therefore, we obtain
\[ p_\ell(g,h)=-\frac{1}{\ell^{h_\ell(r+1)}}\frac{\ell^{r+1}(\ell-1)}{\ell^{r+1}-1}. \qedhere \]
\end{proof}

\section{The order being $k$-free}\label{sect-kfree} 
In this section we prove Theorem \ref{thm-kfree}. 
\begin{proof}[Proof of Theorem \ref{thm-kfree}]
If $\p$ is a prime of $K$ with $\p\notin\mc P(G)$, then $\ord_\p(G)$ is $k$-free if and only if for every rational prime $q$ we have $q^k\nmid\ord_\p(G)$. Therefore, by the inclusion-exclusion principle we have
\[ \pi_{\mc N_k}(x)=\sum_{m\geq1}\mu(m)\pi_{\mc P_{m^k}}^1(x)+\bo{\frac{\sqrt{x}}{\log x}}, \]
where $\pi_{\mc P_{m}}^1(x)$ counts the primes of degree $1$  in $\mc P_m$ with norm up to $x$ (the $O$-term estimates the primes not of degree 1). Notice that for $\pi_{\mc P_{m}}^1(x)$ we may take the same asymptotic formula \eqref{eq-divisible} as  for $\pi_{\mc P_m}(x)$. Then, for $0<a<1$ and $z:=\log^ax$ we have 
\begin{align}
\pi_{\mc N_k}(x) & = \sum_{m\leq z}\mu(m)\pi_{\mc P_{m^k}}^1(x)+\bigg(\sum_{m>z}\pi_{\mc P_{m^k}}^1(x)\bigg)+\bo{\frac{\sqrt{x}}{\log x}} \notag \\
&=\frac{x}{\log x}\beta_{C,k}+\bigg(\frac{x}{\log x}\sum_{m>z}\varrho_{C,m^k}\bigg)+\bigg(\sum_{m>z}\pi_{\mc P_{m^k}}^1(x)\bigg) \label{eq-errors1} \\
& \quad +O\Bigg(\sum_{m\leq z} x\bigg( \frac{(\log\log x)^2}{\log x} \bigg)^{1+\frac{1}{r+2}}\Bigg)+\bo{\frac{\sqrt{x}}{\log x}}\,. \label{eq-errors2}
\end{align}
By Theorem \ref{thm-divisible} we have 
\[  \sum_{m>z}\pi_{\mc P_{m^k}}^1(x)\ll_{F,K,G}\frac{x}{\log x}\sum_{m>z}\varrho_{C,m^k}. \]
By Proposition \ref{prop-convergent} and recalling that $m/\varphi(m)=O_\eta(m^{\eta})$, we have $\varrho_{C,m^k}\ll_\eta 1/m^{k-\eta}$ for every $0<\eta<1$. Hence both $O$-terms in \eqref{eq-errors1} are bounded by
\begin{equation}\label{error1}
\frac{x}{\log x}\sum_{m>z}\frac{1}{m^{k-\eta}}\ll_\eta \frac{x}{(\log x )^{1+a(k-1-\eta)}} .
\end{equation}

Next, we can bound the first error term in \eqref{eq-errors2} by
 \begin{equation}\label{error2}
(\log x)^{a}\cdot x\bigg( \frac{(\log\log x)^2}{\log x} \bigg)^{1+\frac{1}{r+2}}\leq \frac{x(\log\log x)^{3}}{(\log x)^{1+\frac{1}{r+2}-a}} ,
 \end{equation}
and we may choose $a<\frac{1}{(r+2)(k-\eta)}$ so that \eqref{error2} can be bounded by \eqref{error1}. Thus, we may decrease the exponent in the denominator of \eqref{error1} until $1+\frac{k-1}{(r+2)k+1}$, taking, e.g.\ $\eta=1/(r+4)$.
Collecting the errors yields the result.
\end{proof}

\begin{rem}
In the context of Theorem \ref{thm-kfree}, the case of groups with torsion is straightforward: if $G'$ is a finitely generated subgroup of $K^\times$ with torsion of order $t$, and $G=G'/\group{\zeta_t}$, then the density of the set $\mc N_{k,G'}$ (i.e.\ $\mc N_k$ defined for the group $G'$) is equal to the density of $\mc N_{k,G}$ if $t$ is $k$-free, and it is $0$ otherwise.
\end{rem}

Next we prove an explicit formula for the density $\beta_{C,k}$ of Theorem \ref{thm-kfree} if the condition on the Frobenius is trivial, and in this case we simply write $\beta_k$.

\begin{thm}\label{cor-kfree}
Let $K$ be a number field and let $G$ be a finitely generated and torsion-free subgroup of $K^\times$ of positive rank $r$. Let $k\geq2$ and let $\beta_k$ be the natural density of the set of primes $\p$ of $K$ such that  $\ord_\p(G)$ is $k$-free (where $\p\notin\mc P(G)$).
Then there is an integer $z$, which depends only on $K$ and $G$, such that 
\begin{equation}\label{eqcor-kfree0}
\beta_k=\prod_{\ell\nmid z}\Big(1-\frac{\ell^r-1}{(\ell-1)(\ell^{r+1}-1)\ell^{k-2}}\Big)\cdot  
\sum_{\substack{g\mid z \\ (\rad(g)^k,z)\mid g}}\sum_{h\mid g}\frac{p(g,h)}{[K_{g,h}:K]},
\end{equation}
where we set $p(g,h)=0$ if and only if at least one of the following conditions is satisfied:
\begin{itemize}
\item there is $\ell\mid g$, $\ell\nmid h$ and $v_\ell(g)\neq v_\ell((\ell^k,z))$,
\item there is $\ell\mid h$ such that $v_\ell(g/h)>k$, or $v_\ell(z/g)>0$ and $v_\ell(g/h)<k-1$;
\end{itemize}
else we define $p(g,h)$ to be
\[
\frac{g}{h\rad(g)^k}\cdot 
\prod_{\substack{\ell\mid g,\, \ell\nmid h \\ v_\ell(\frac{g}{(\ell^k,z)})=0 \\ \text{or } \ell\mid h,\, v_\ell(\frac{g}{h})=k}}  (-1)\cdot
\prod_{\substack{\ell\mid h \\ v_\ell(\frac{z}{g})>0 \\ v_\ell(\frac{g}{h})=k-1}}\ell \cdot 
\prod_{\substack{\ell\mid h \\ v_\ell(\frac{z}{g})=0 \\ 0< v_\ell(\frac{g}{h})< k}}(\ell-1) \cdot
\prod_{\substack{\ell\mid h \\  v_\ell(\frac{z}{h})=0}} \frac{\ell^{r+1}(\ell-1)}{\ell^{r+1}-1} .  \]
\end{thm}

In Section \ref{sect-ex} we point out that the obtained formula for $\beta_k$ can be rewritten as a rational multiple of a certain absolute constant $A_{k,r}$, see \eqref{Akr}, depending only on the parameters $k$ and $r$, which we evaluate in Table \ref{table-constants}.

\begin{proof}
Applying \cite[Theorem 1.1]{PST} as in the proof of Theorem \ref{thm-dens-formula} we obtain
\begin{align}
\beta_k &=\sum_{m\geq1}\sum_{\substack{n\mid m^\infty \\ d\mid m}}\frac{\mu(m)\mu(d)}{[K_{nm^k,dn}:K]} \notag \\
 &= \sum_{\substack{g\mid z\\ h\mid g}} \frac{\varphi(g)h^r}{[K_{g,h}:K]}\sum_{m\geq1}\, \sum_{\substack{n\mid m^\infty \\ (nm^k,z)=g}}\, \sum_{\substack{d\mid m \\ (dn,z)=h}}\frac{\mu(m)\mu(d)}{n^{r+1}d^r\varphi(m^k)} \notag \\
 &= \sum_{\substack{g\mid z\\ h\mid g}} \frac{\varphi(g)h^r}{[K_{g,h}:K]}\prod_{\ell} p_\ell(g,h) \label{eqcor-kfree1}
\end{align}
where for $\ell\nmid z$ (and hence $\ell\nmid g$) we have
\[ p_\ell(g,h)= 1-\frac{1}{\varphi(\ell^k)}\sum_{s\geq 0}\ \sum_{e\in\{0,1\}}  \frac{\mu(\ell^e)}{\ell^{s(r+1)+er}}=1-\frac{1}{\varphi(\ell^{k})}\frac{\ell(\ell^r-1)}{\ell^{r+1}-1}, \]
and for $\ell\mid z$, $\ell\nmid g$ we have $p_\ell(g,h)=1$, whereas for $\ell\mid g$, setting $z_\ell:=v_\ell(z)$ and similarly for $g_\ell,h_\ell$,  we have
\[
 p_\ell(g,h)= -\frac{1}{\varphi(\ell^k)}\sum_{\substack{s\geq 0 \\ \min(k+s,z_\ell)=g_\ell}}\ \sum_{\substack{e\in\{0,1\} \\ \min(s+e,z_\ell)=h_\ell}}  \frac{\mu(\ell^e)}{\ell^{s(r+1)+er}}.
\]

We take $p(g,h):=\varphi(g)h^r\prod_{\ell\mid g}p_\ell(g,h)$ (and we will make use of the identity $\varphi(g)/\varphi(\rad(g)^k)=g/\rad(g)^k$). Let us compute $p_\ell(g,h)$ depending on the prime $\ell\mid g$. If $g_\ell< \min(k,z_\ell)$, then $p_\ell(g,h)=0$, so we may restrict the sum in \eqref{eqcor-kfree1} to the divisors $g$ such that $(\rad(g)^k,z)\mid g$.

\emph{Case 1: $\ell\nmid h$.} The conditions on the indices hold only for $s=e=0$. Thus, if $g_\ell=\min(k,z_\ell)$, then  $p_\ell(g,h)=-1/\varphi(\ell^{k})$, otherwise $p_\ell(g,h)=0$.

\emph{Case 2: $\ell\mid h$ and $h_\ell=g_\ell=z_\ell$.} The sums reduce to the indices $s,e$ such that $s+e\geq h_\ell$ (recall that $k\geq2$). Hence, we have
\[ p_\ell(g,h)= \frac{1}{\varphi(\ell^{k})\ell^{(r+1)h_\ell}}\frac{\ell^{r+1}(\ell-1)}{\ell^{r+1}-1}.  \]

\emph{Case 3: $\ell\mid h$ and $h_\ell<g_\ell=z_\ell$.} The conditions on the indices become $s+e=h_\ell$ and $k+s\geq g_\ell$. Hence, we have: $p_\ell(g,h)=0$ if $g_\ell-h_\ell> k$; $p_\ell(g,h)=-1/(\varphi(\ell^{k})\ell^{(r+1)h_\ell})$ if $g_\ell-h_\ell= k$; $p_\ell(g,h)=(\ell-1)/(\varphi(\ell^{k})\ell^{(r+1)h_\ell})$ if $g_\ell-h_\ell< k$.

\emph{Case 4: $\ell\mid h$ and $g_\ell<z_\ell$.} The conditions on the indices become $s+e=h_\ell$ and $k+s=g_\ell$. Thus, we have: if $g_\ell-h_\ell=k$, then $p_\ell(g,h)=-1/(\varphi(\ell^{k})\ell^{(r+1)h_\ell})$; if $g_\ell-h_\ell=k-1$, then $p_\ell(g,h)=\ell/(\varphi(\ell^{k})\ell^{(r+1)h_\ell})$; otherwise, $p_\ell(g,h)=0$.
\end{proof}

\section{Prescribing valuations for the order}\label{sect-val}
In this section we apply Theorem \ref{thm-divisible} to prove an asymptotic formula for the number of primes $\p$ of $K$ for which the order of $(G\bmod\p)$ has  prescribed $\ell$-adic valuations for finitely many given primes $\ell$.

\begin{thm}\label{thm-val}
Let $K$ be a number field and let $G$ be a finitely generated and  torsion-free subgroup of $K^\times$ of positive rank $r$. Let $F/K$ be a finite Galois extension, and let $C$ be a conjugacy-stable subset of $\Gal(F/K)$.
Consider finitely many prime numbers $\ell$, and for each of them fix a nonnegative integer $a_\ell$. Set $k=\prod\ell$ and $m=\prod\ell^{a_\ell}$, where $\ell$ runs through the considered primes. Consider the set of primes of $K$ given by
\[ \mc V =\set{ \p: v_{\ell}(\ord_\p(G))=a_\ell \ \forall \ell\mid k,\,\begin{pmatrix}\p \\ F/K  \end{pmatrix}\subseteq C },  \]
where we are assuming that $\p\notin\mc P(G)$ and $\p$ does not ramify in $F$.
Then  we have
\[
\pi_{\mc V}(x)=\frac{x}{\log x}\sum_{f\mid k}\mu(f)\varrho_{C,mf} +O\left( \tau(k)x\bigg( \frac{(\log\log x)^2}{\log x} \bigg)^{1+\frac{1}{r+2}} \right), 
\]
where, for $t\geq1$, $\varrho_{C,t}$ is as in \eqref{eq-div-dens}, hence the set $\mc V$ has natural density
\[
\gamma_{C,k,m}:=\sum_{f\mid k}\,\sum_{n\mid (fm)^\infty}\,\sum_{d\mid fm}\frac{\mu(f)\mu(d)c(fmn,dn)}{[F_{fmn,dn}:K]}
\]
(with $c(a,b)$ as in \eqref{coeff-c1}). The constant implied by the $O$-term depends only on $F$, $K$, $G$.
\end{thm}
It follows from Proposition \ref{prop-convergent} that the series $\gamma_{C,k,m}$ is convergent.

\begin{proof}
We must have that $\ord_\p(G)$ is divisible by $m$ and not by $m\ell$ for any prime factor $\ell$ of $k$. Hence applying the inclusion-exclusion principle and Theorem \ref{thm-divisible} we obtain the desired formula.
\end{proof}

Notice that $\gamma_{C,k,m}$ is given by a finite sum of terms of the form $\pm\varrho_{C,t}$. 
In the following  we provide an explicit formula for the special case of trivial condition on the Frobenius.

\begin{thm}\label{cor-val}
Let $K$ be a number field and let $G$ be a torsion-free subgroup of $K^\times$ of positive rank $r$.
Let $k,m\geq1$ be integers with $k$ squarefree and $\rad(m)\mid k$. Let $\gamma_{k,m}$ be the natural density  of the set of primes $\p$ of $K$ such that $v_\ell(\ord_\p(G))=v_\ell(m)$ for all $\ell\mid k$
(where $\p\notin\mc P(G)$).
Then there is an integer $z$, which depends only on $K$ and $G$, such that 
\[
\gamma_{k,m}=\frac{1}{\varphi(m)}
\prod_{\substack{\ell\mid m \\ \ell\nmid z}}\frac{(\ell-1)(\ell^r-1)}{\ell^{r+1}-1}
\prod_{\substack{\ell\mid k \\ \ell\nmid mz}}\bigg(1-\frac{\ell(\ell^r-1)}{(\ell^{r+1}-1)(\ell-1)}\bigg)
  \sum_{\substack{g\mid z \\ (m,z)\mid g\\ \rad(g)\mid k}}\sum_{h\mid g}\frac{p(g,h)}{[K_{g,h}:K]},
\]
where, for $h\mid g$, we set $p(g,h)=0$ if and only if at least one of the following conditions holds: 
\begin{itemize}
\item there is $\ell\mid (k,g)$, $\ell\nmid m$, such that $v_\ell(g/h)>1$,
\item there is $\ell\mid (g,m)$, $\ell\nmid h$, such that we have $v_\ell(g)\notin\{v_\ell(z),v_\ell(m),v_\ell(m)+1  \}$, or $v_\ell(g)=v_\ell(z)>v_\ell(m)+1$,
\item there is $\ell\mid (h,m)$, such that we have $v_\ell(g/h)>v_\ell(m)+1$, or we have  $v_\ell(z/g)>0$ and  $v_\ell(g/h)<v_\ell(m)-1$;
\end{itemize}
else we define $p(g,h)=\frac{\varphi(g)}{h}\cdot q_1(g,h)q_2(g,h)$, with
\[ q_1(g,h) = 
\prod_{\substack{\ell\mid \frac{g}{h} \\ \ell\nmid m}}\frac{-1}{\ell-1} \cdot
\prod_{\substack{\ell\mid h,\, \ell\nmid m \\ v_\ell(\frac{z}{h})=0}}
\frac{\ell^{r+1}}{\ell^{r+1}-1} \cdot 
 \prod_{\substack{\ell\mid h,\, \ell\nmid m \\ v_\ell(h)=v_\ell(g)<v_\ell(z)}} \frac{\ell}{\ell-1} \]
(the primes involved in these products are coprime with $m$), and
\begin{multline*}
q_2(g,h) =  
\prod_{\substack{\ell\mid (h,m) \\ v_\ell(\frac{z}{h})= 0}} \frac{-\ell^{r}(\ell-1)^2}{\ell^{r+1}-1}\cdot
\prod_{\substack{\ell\mid (g,m) \\ v_\ell(\frac{g}{h})=v_\ell(m)+1}}\frac{-1}{\ell}\cdot
\prod_{\substack{\ell\mid g,\, \ell\nmid h  \\  v_\ell(g)=v_\ell(z)\leq v_\ell(m) }}\frac{\ell-1}{\ell}\,\cdot 
\\
\cdot\prod_{\substack{\ell\mid (h,m) \\ v_\ell(\frac{z}{g})>0 \\ v_\ell(\frac{g}{h})=v_\ell(m)}} 2
\cdot\prod_{\substack{\ell\mid h,\, v_\ell(\frac{z}{g})=0 \\ 0< v_\ell(\frac{g}{h})< v_\ell(m)}}\frac{-(\ell-1)^2}{\ell}\cdot
\prod_{\substack{\ell\mid h,\, v_\ell(\frac{z}{g})=0 \\  v_\ell(\frac{g}{h})= v_\ell(m)>0}}\frac{2\ell-1}{\ell}\cdot
\prod_{\substack{\ell\mid h,\, v_\ell(\frac{z}{g})>0 \\ v_\ell(\frac{g}{h})=v_\ell(m)-1}} -\ell
\end{multline*}
(the primes involved in these products are prime factors of $m$).
\end{thm}

The proof is similar to that of Theorem \ref{thm-dens-formula}: one needs to first apply \cite[Theorem 1.1]{PST}, then transform the obtained inner sums into a product on the prime factors $\ell$ of $k$, and compute these through a certain case distinction. For the convenience of the reader, we give all details below.

\begin{proof}
Applying \cite[Theorem 1.1]{PST} as  in the proof of Theorem \ref{thm-dens-formula}, we obtain
\begin{align*}
\gamma_{k,m} & = \sum_{f\mid k}\ \sum_{n\mid (fm)^\infty}\ \sum_{d\mid fm}\frac{\mu(f)\mu(d)}{[K_{fmn,dn}:K]} \\
&  = \sum_{\substack{g\mid z\\ h\mid g}} \frac{\varphi(g)h^r}{[K_{g,h}:K]}
 \ \sum_{f\mid k}\ \sum_{\substack{n\mid (fm)^\infty \\ (fmn,z)=g}} \ \sum_{\substack{d\mid fm \\ (dn,z)=h}}\frac{\mu(f)\mu(d)}{\varphi(fmn)(nd)^r} .
\end{align*}
Write $m=\prod_{\ell\mid k}\ell^{a_\ell}$ with $a_\ell\geq0$.
Notice that the inner sums do not vanish only if $(m,z)\mid g$ and $\rad(g)\mid k$. Thus we reduce the formula to
\[ \gamma_{k,m} =\frac{1}{\varphi(m)}\cdot \sum_{\substack{g\mid z \\ (m,z)\mid g \\ \rad(g)\mid k}}\sum_{h\mid g} \frac{\varphi(g)h^r}{[K_{g,h}:K]}\cdot\prod_{\ell\mid k} p_\ell(g,h), \]
where for $\ell\mid k$ and $a_\ell=0$ we define $p_\ell(g,h)$ according to a case distinction on the prime factors $\ell$ of $k$. Set $z_\ell:=v_\ell(z)$, and similarly for $g_\ell,h_\ell$.

\emph{Case 1: $\ell\mid k$ and $a_\ell=0$.} If $\ell\nmid z$, then we set
\[ p_\ell(g,h):=1-\frac{1}{\ell-1}\sum_{s\geq0}\ \sum_{t\in\{0,1\}}  \frac{\mu(\ell^t)}{\ell^{s(r+1)+tr}}
=1-\frac{\ell(\ell^r-1)}{(\ell-1)(\ell^{r+1}-1)};  \]
if $\ell\mid z$ and $\ell\nmid g$, then $p_\ell(g,h):=1$; if $\ell\mid g$, then we set
\[ p_\ell(g,h):=-\frac{1}{\ell-1}\sum_{\substack{s\geq0\\ \min(s+1,z_\ell)=g_\ell }}\ \sum_{\substack{t\in\{0,1\} \\ \min(s+t,z_\ell)=h_\ell}}  \frac{\mu(\ell^t)}{\ell^{s(r+1)+tr}}. \]
We take $q_1(g,h):=\prod_{\ell\mid (k,g),\,a_\ell=0}p_\ell(g,h)\ell^{h_\ell(r+1)}$, and we compute $p_\ell(g,h)$ depending on the prime factor $\ell$ of $(k,g)$.

\emph{Case 1.1: $\ell\nmid h$.} We must have $s=t=0$, hence $p_\ell(g,h)=-1/(\ell-1)$ if $g_\ell=1$, otherwise $p_\ell(g,h)=0$.

\emph{Case 1.2: $\ell\mid h$ and  $h_\ell=g_\ell=z_\ell$.} The conditions on the indices hold only for $s+t\geq z_\ell$. We obtain $p_\ell(g,h)=\frac{1}{\ell^{h_\ell(r+1)}}\frac{\ell^{r+1}}{\ell^{r+1}-1}$.

\emph{Case 1.3: $\ell\mid h$ and  $h_\ell<g_\ell=z_\ell$.} We must have $s+t=h_\ell$ and $s+1\geq z_\ell$, hence $p_\ell(g,h)=-1/(\ell^{h_\ell(r+1)}(\ell-1))$ if $g_\ell-h_\ell=1$, and $p_\ell(g,h)=0$ otherwise.

\emph{Case 1.4: $\ell\mid h$ and  $g_\ell<z_\ell$.} We must have $s=g_\ell-1$. Thus, we obtain: $p_\ell(g,h)=-1/(\ell^{h_\ell(r+1)}(\ell-1))$ if $g_\ell-h_\ell=1$; $p_\ell(g,h)=\ell/(\ell^{h_\ell(r+1)}(\ell-1))$ if $g_\ell=h_\ell$; $p_\ell(g,h)=0$ otherwise.

\emph{Case 2: $\ell\mid k$ and $a_\ell>0$.} Recall that $g$ is such that $g_\ell\geq \min(a_\ell,z_\ell)$. We set
\[
 p_\ell(g,h):=\sum_{e\in\{0,1\}}\ \sum_{\substack{ s\geq0\\ \min(e+a_\ell+s,z_\ell)=g_\ell}}\ \sum_{\substack{t\in\{0,1\} \\ \min(s+t,z_\ell)=h_\ell} } \frac{\mu(\ell^e)\mu(\ell^t)}{\ell^{s(r+1) +e+tr}}. 
\]
If $\ell\nmid z$, then we have  $p_\ell(g,h)=\frac{(\ell-1)(\ell^r-1)}{\ell^{r+1}-1}$.
Before dealing with the case $\ell\mid z$, we set $q_2(g,h):=\prod_{\ell\mid (k,z),\,a_\ell>0}p_\ell(g,h)\ell^{h_\ell(r+1)}$, and we compute $p_\ell(g,h)$ depending on the prime factor $\ell$ of $(k,z)$.

\emph{Case 2.1: $\ell\nmid h$.} We must have $s=t=0$ and hence we have: $p_\ell(g,h)=0$ if $g_\ell\notin \{z_\ell,a_\ell,a_\ell+1\}$, or $g_\ell=z_\ell>a_\ell+1$; $p_\ell(g,h)=1$ if $g_\ell=a_\ell<z_\ell$; $p_\ell(g,h)=-1/\ell$ if $g_\ell=a_\ell+1$; $p_\ell(g,h)=1-1/\ell$ if $g_\ell=z_\ell\leq a_\ell$.

\emph{Case 2.2: $\ell\mid h$ and $h_\ell=g_\ell=z_\ell$.} The conditions on the indices hold only for $s\geq h_\ell$, or $s=h_\ell-1$ and $t=1$ (as $a_\ell\geq1$). We obtain 
\[ p_\ell(g,h)=\Big(1-\frac{1}{\ell}\Big)\Big(\Big(1-\frac{1}{\ell^r}\Big)\frac{1}{\ell^{h_\ell(r+1)}}\sum_{s\geq 0}\frac{1}{\ell^{s(r+1)}}-\frac{\ell}{\ell^{h_\ell}}  \Big)=-\frac{\ell^r(\ell-1)^2}{\ell^{h_\ell(r+1)}(\ell^{r+1}-1)}. \]

\emph{Case 2.3: $\ell\mid h$ and $h_\ell<g_\ell=z_\ell$}. Then the conditions are satisfied only if  $s+e+a_\ell\geq z_\ell$. We deduce that $p_\ell(g,h)$ equals: $-(\ell-1)^2/\ell^{h_\ell(r+1)+1}$ if $g_\ell-h_\ell<a_\ell$; $(2\ell-1)/\ell^{h_\ell(r+1)+1}$ if $g_\ell-h_\ell=a_\ell$; $-1/\ell^{h_\ell(r+1)+1}$ if $g_\ell-h_\ell= a_\ell+1$; $0$ if $g_\ell-h_\ell> a_\ell+1$.

\emph{Case 2.4: $\ell\mid h$ and $g_\ell<z_\ell$}. The conditions are satisfied  only if $a_\ell+s+e=g_\ell$. Thus, $p_\ell(g,h)$ equals: $0$ if $g_\ell-h_\ell\notin \{a_\ell-1,a_\ell,a_\ell+1\}$; $-\ell/\ell^{h_\ell(r+1)}$  if $g_\ell-h_\ell=a_\ell-1$; $2/\ell^{h_\ell(r+1)}$ if $g_\ell-h_\ell=a_\ell$; $-1/\ell^{h_\ell(r+1)+1}$  if $g_\ell-h_\ell=a_\ell+1$.

Finally, take $p(g,h):=\frac{\varphi(g)}{h}\cdot q_1(g,h)q_2(g,h)$.
\end{proof}

\begin{cor}
Let $K$ be a number field and let $G$ be a torsion-free subgroup of $K^\times$ of positive rank $r$.
Let $k\geq1$ be a squarefree integer. The natural density of the set of primes $\p$ of $K$ such that $\ord_\p(G)$ is coprime to $k$ and $(\p,F/K)\subseteq C$ (where $\p\notin\mc P(G)$ and $\p$ does not ramify in $F$) is given by
\[ \sum_{f\mid k}\mu(f)\varrho_{C,f}, \]
where, for $t\geq1$, $\varrho_{C,t}$ is as in \eqref{eq-div-dens}.
Moreover, supposing that the condition on the Frobenius is trivial, there is an integer $z$, which depends only on $K$ and $G$, such that this density equals
\[
\prod_{\substack{\ell\mid k \\ \ell\nmid z}}\bigg(1-\frac{\ell(\ell^r-1)}{(\ell^{r+1}-1)(\ell-1)}\bigg)
\cdot  \sum_{\substack{g\mid z \\  \rad(g)\mid k}}\sum_{h\mid g }
\mu\left( \frac{g}{h} \right)
\frac{p(g,h)}{[K_{g,h}:K]},
\]
where, for $h\mid g$, we set 
\[ p(g,h) = \frac{\varphi(g)}{\varphi(g/h)h}\cdot
\prod_{\substack{\ell\mid h \\ \ell \nmid \frac{z}{h}}} \frac{\ell^{r+1}}{\ell^{r+1}-1} \cdot 
 \prod_{\substack{\ell\mid h \\ v_\ell(h)=v_\ell(g)<v_\ell(z)}} \frac{\ell}{\ell-1}. \]
\end{cor}

\begin{proof}
The first part of the statement follows directly from Theorem \ref{thm-divisible} and it is a special case of Theorem \ref{thm-val}.
As for the second part, it suffices to take $m=1$ in Theorem \ref{cor-val}. Clearly we have $q_2(g,h)=1$. We obtain that $p(g,h)=0$ if and only if $g/h$ is not squarefree, and this together with the factor $\prod_{\ell\mid \frac{g}{h}}(-1)$ in $q_1(g,h)$ yields the term $\mu(g/h)$ in the formula. Also, the factor $\prod_{\ell\mid \frac{g}{h}}1/(\ell-1)$ in $q_1(g,h)$ is equal to $1/\varphi(g/h)$.
\end{proof}

\section{Conditional results assuming GRH}\label{sect-GRH}
In this section we show how Theorems \ref{thm-divisible} and \ref{thm-kfree} can be improved if we assume  GRH for the Dedekind zeta functions of number fields of the type $K_{m,n}$. In fact, in this case we can make use of the stronger version of the Chebotarev density theorem, namely \cite[Th\'eor\`eme 4]{serre} or \cite[Theorem 2]{zieg}, and we obtain smaller error terms. Let us first apply this theorem to cyclotomic-Kummer extensions of $K$.

\begin{lem}\label{CDT-GRH}
Assume GRH.
Let $F/K$, $G$, $C$ and   $\pi_{F_{m,n}/K,C_{m,n}}(x)$ be as in Theorem \ref{thm-chebTZ-kummer}. Then we have
\begin{equation}\label{CDT-GRH-eq}
\pi_{F_{m,n}/K,C_{m,n}}(x)=\frac{\av{C_{m,n}}}{[F_{m,n}:K]}\Li(x)+O_{F,G}\left( \sqrt{x}\log(mx) \right).
\end{equation}
\end{lem}

\begin{proof}
Applying \cite[Th\'eor\`eme 4]{serre} we have
\[ \frac{\av{C_{m,n}}}{[F_{m,n}:K]}\Big(\Li(x)+\bo{\sqrt{x}\log\big( \av{d_{F_{m,n}}}x^{[F_{m,n}:\Q]} \big)}\Big). \]
Recalling that $\av{C_{m,n}}\leq[F:K]$ and applying Proposition \ref{prop-disc}, we can reduce the error term to
\[ O_F\bigg(\sqrt{x}\cdot \frac{\log\av{d_{F_{m,n}}}}{[F_{m,n}:\Q]}+\sqrt{x}\log x\bigg)=O_{F,G}\left( \sqrt{x}\log m + \sqrt{x}\log x\right). \qedhere \]
\end{proof}

\begin{thm}\label{thm-divisGRH}
Assume GRH.
With the setup of Theorem \ref{thm-divisible},  we have
\[
\pi_{\mc P_m}(x)=\Li(x)\varrho_{C,m} +O_{F,K,G}( x^{3/4}\log x ).
\]
\end{thm}

\begin{proof}
We start as in the proof of Theorem \ref{thm-divisible}. Recall \eqref{eq-B}, where $\mc B_n$ was defined in \eqref{def_Bn}. Applying first Lemma \ref{lem-key-id}, and then Lemma \ref{CDT-GRH} to the functions $\pi^1_{mn,dn}(x)$ (notice that \eqref{CDT-GRH-eq} also holds if we restrict to the primes of $K$ of degree $1$), setting $y:=x^{1/4}$ we obtain: 
\begin{align*}
\pi_{\mc P_m}(x) = &\ \Li(x)\varrho_{C,m}+O\Bigg(\sum_{n\leq y/m}\,\sum_{d\mid m}\sqrt{x}\log(mnx)\Bigg)+\bo{\frac{\sqrt{x}}{\log x}}  \\
& \ +O\Bigg(\Li(x)\sum_{\substack{n\mid m^\infty \\ nm>y}}\,\sum_{d\mid m}\frac{\mu(d)c(mn,dn)}{[F_{mn,dn}:K]}\Bigg) + O\Bigg(\sum_{\substack{n\mid m^\infty\\ nm> y}}\pi_{\mc B_n}(x)\Bigg) \\
=&\  \Li(x)\varrho_{C,m}+\bo{y\sqrt{x}\log x}+\bo{\frac{x\log\log x}{y\log x}}+O\Bigg(\frac{\Li(x)}{\varphi(m)}\sum_{n>y/m}\frac{1}{n^2}\Bigg).
\end{align*}
For the last error term, we used that 
\[ \pi_{\mc B_n}(x)\leq\pi_{K_{mn,n}/K,\id}(x) 
\ll_{K,G} \frac{1}{[K_{mn,n}:K]}\Li(x) \]
by Lemma \ref{CDT-GRH}, and the estimate $[K_{mn,n}:K]\gg_{K,G} \varphi(m)n^2$; see \cite[Theorem 3.1]{PS1}. All error terms are hence estimated by $x^{3/4}\log x$.
\end{proof}

\begin{cor}
Assume GRH.
With the setup of Theorem \ref{thm-kfree},  we have
\[
\pi_{\mc N_k}(x)=\Li(x)\beta_{C,k} +O_{F,K,G}\Bigg( \frac{x^{1-\frac{k-1}{4k+1}}}{\log x} \Bigg).
\]
Moreover, with the setup of Theorem \ref{thm-val}, we have 
\[
\pi_{\mc V}(x)=\Li(x)\gamma_{C,k,m} +O_{F,K,G}\big(\tau(k)x^{3/4}\log x\big).
\]
\end{cor}

\begin{proof}
As for the first assertion, it is sufficient to follow the proof of Theorem \ref{thm-kfree},  making use of Theorem \ref{thm-divisGRH} instead of Theorem \ref{thm-divisible}. Here we set $z=x^a$ for some $0<a<1/4$. This yields
\[ \pi_{\mc N_k}(x)=\Li(x)\beta_{C,k} + O\left(\frac{x}{x^{a(k-1-\eta)}\log x}+x^{3/4+a}\log x  \right). \]
Taking $a<\frac{1}{4(k-\eta)}$ we can bound the second error term with the first one. Hence we may decrease the exponent $a(k-1-\eta)$ to $\frac{k-1}{4k+1}$ (e.g.\ taking $\eta=1/10$), resulting in the error term of the statement.
The second assertion is a direct consequence of Theorem \ref{thm-divisGRH}.
\end{proof}

\section{Numerical examples}\label{sect-ex}
In this section we provide several examples of densities computed with the formulas of Theorems \ref{thm-dens-formula}, \ref{cor-kfree} and \ref{cor-val}. All values have been verified with SageMath \cite{sage} by computing the approximated density that considers only primes up to a certain bound. In particular, we have tested these formulas for $K$ and $G$ as in the several numerical examples from \cite{PeruccaKummer, Debry, PeruccaWIN, pappa, pappa-sf} (notice that in \cite[Table 3, left side]{Debry} the density for the fifth and seventh entries should both read $121/960$).

Let $K$ be a number field and $G$ a finitely generated subgroup of $K^\times$. Recall the notation $\varrho_{m}$ introduced in  Theorem \ref{thm-dens-formula}. In Tables \ref{table-Q-div}--\ref{table-Q4-div} we provide several examples of densities $\varrho_{m}$.

\begin{table}[ht]
\centering
\caption{Examples of densities $\varrho_{m}$ with $K=\Q$}
\label{table-Q-div}
\scriptsize
\bgroup
\def\arraystretch{1.2}
\begin{tabular}{c||ccccccccc}
$G$  & $\varrho_{2}$ &  $\varrho_{3}$ & $\varrho_{4}$ & $\varrho_{6}$ & $\varrho_{9}$ & $\varrho_{12}$ & $\varrho_{16}$  & $\varrho_{27}$ \\ [0.5ex] 
\hline
$\group{2}$ & $17/24$ & $3/8$ & $5/12$ & $17/64$ & $1/8$ & $5/32$ & $1/24$  &$1/24$ \\
$\group{16}$ & $1/12$ & $3/8$ & $1/24$ & $1/32$ & $1/8$ & $1/64$ & $1/96$& $1/24$  \\
$\group{3}$ & $2/3$ & $3/8$ & $1/3$ & $5/16$ & $1/8$ & $1/16$ & $1/12$& $1/24$  \\
$\group{27}$ & $2/3$ & $1/8$ & $1/3$ & $5/48$ & $1/24$ & $1/48$ & $1/12$& $1/72$  \\
$\group{2,3}$ & $195/224$ & $6/13$ & $27/56$ & $333/728$ & $2/13$ & $3/14$ & $5/56$&  $2/39$ \\
$\group{16,27}$ & $75/112$ & $5/13$ & $75/224$ & $235/728$ & $5/39$ & $95/1456$ & $75/896$&  $5/117$ \\  
$\group{2,27,25}$ & $839/960$ & $37/80$ & $59/120$ & $17723/38400$
 & $37/240$ & $1073/4800$ & $11/120$  & $37/720$ \\ [0.5ex] 
\end{tabular}
\egroup
\end{table}

\begin{table}[ht]
\caption{Examples of densities $\varrho_{m}$ with $K=\Q(\zeta_3)$}
\label{table-Q3-div}
\centering\scriptsize
\bgroup
\def\arraystretch{1.2}
\begin{tabular}{c||ccccccccc}
$G$  & $\varrho_{2}$ &  $\varrho_{3}$ & $\varrho_{4}$ & $\varrho_{6}$ & $\varrho_{9}$ & $\varrho_{12}$ & $\varrho_{16}$ &  $\varrho_{27}$ \\ [0.5ex] 
\hline
$\group{2}$ & $17/24$ & $3/4$ & $5/12$ & $17/32$ & $1/4$ & $5/16$ & $1/24$  & $1/12$ \\
$\group{16}$ &  $1/12$ & $3/4$ & $1/24$ & $1/16$ & $1/4$ &  $1/32$& $1/96$  & $1/12$ \\
$\group{3}$ &  $5/6$ & $3/4$ & $1/6$ & $5/8$ & $1/4$ & $1/8$ & $1/24$  & $1/12$\\
$\group{27}$ & $5/6$  & $1/4$ & $1/6$ & $5/24$ & $1/12$ & $1/24$ & $1/24$  & $1/36$\\
$\group{2,3}$ & $111/112$  & $12/13$ & $13/28$ & $333/364$ & $4/13$ & $3/7$ & $3/56$  & $4/39$\\
$\group{16,27}$ & $47/56$  & $10/13$ & $19/112$ & $235/364$ & $10/39$ & $95/728$ & $19/448$  & $10/117$ \\
$\group{2,27,25}$ &  $479/480$ & $37/40$ & $29/60$ & $17723/19200$ & $37/120$ & $1073/2400$ & $7/120$  & $37/360$\\ [0.5ex] 
\end{tabular}
\egroup
\end{table}

\begin{table}[ht]
\caption{Examples of densities $\varrho_{m}$ with $K=\Q(\zeta_4,\sqrt{3})$}
\label{table-multi-div}
\centering\scriptsize
\bgroup
\def\arraystretch{1.2}
\begin{tabular}{c||ccccccccc}
$G$  & $\varrho_{2}$ &  $\varrho_{3}$ & $\varrho_{4}$ & $\varrho_{6}$ & $\varrho_{9}$ & $\varrho_{12}$ & $\varrho_{16}$ &  $\varrho_{27}$ \\ [0.5ex] 
\hline
$\group{2}$ & $11/12$ & $3/4$ & $5/6$ & $11/16$ & $1/4$ & $5/8$ & $1/12$ & $1/12$ \\
$\group{16}$ &  $1/6$ & $3/4$ & $1/12$ & $1/8$ & $1/4$ & $1/16$ & $1/48$  & $1/12$ \\
$\group{3}$ &  $2/3$ & $3/4$ & $1/3$ & $1/2$ &$1/4$  & $1/4$ & $1/12$  & $1/12$ \\
$\group{27}$ & $2/3$  & $1/4$ & $1/3$ & $1/6$ & $1/12$ & $1/12$ & $1/12$  & $1/36$ \\
$\group{2,3}$ & $55/56$  & $12/13$ & $13/14$ & $165/182$ & $4/13$ & $6/7$ & $3/28$  & $4/39$ \\
$\group{16,27}$ & $19/28$  & $10/13$ & $19/56$ & $95/182$ & $10/39$ & $95/364$ &  $19/224$ & $10/117$ \\
$\group{2,27,25}$ &  $239/240$  & $37/40$ & $29/30$ & $8843/9600$ & $37/120$ & $1073/1200$ &  $7/60$ & $37/360$ \\ [0.5ex] 
\end{tabular}
\egroup
\end{table}

\begin{table}[ht]
\caption{Examples of densities $\varrho_{m}$ with $K=\Q(\zeta_4)$}
\label{table-Q4-div}
\centering\scriptsize
\bgroup
\def\arraystretch{1.2}
\begin{tabular}{c||ccccccccc}
$G$  & $\varrho_{2}$ &  $\varrho_{3}$ & $\varrho_{4}$ & $\varrho_{6}$ & $\varrho_{9}$ & $\varrho_{12}$ & $\varrho_{16}$ &  $\varrho_{27}$ \\ [0.5ex] 
\hline
$\group{2\zeta_4}$ & $2/3$ & $3/8$ & $1/3$ & $1/4$ & $1/8$ & $1/8$ & $1/12$  & $1/24$ \\
$\group{16\zeta_4}$ & $47/48$  & $3/8$ & $23/24$ & $47/128$ & $1/8$ & $23/64$ & $1/48$  & $1/24$ \\
$\group{3\zeta_4}$ & $5/6$  & $3/8$ & $2/3$ & $11/32$ & $1/8$ & $5/16$ & $1/6$  & $1/24$ \\
$\group{27\zeta_4}$ &  $5/6$ & $1/8$ & $2/3$ & $11/96$ & $1/24$ & $5/48$ &  $1/6$ &  $1/72$\\
$\group{2\zeta_4,3\zeta_4}$ & $13/14$  & $6/13$ & $5/7$ & $165/364$ & $2/13$ & $3/7$ &  $5/28$ & $2/39$ \\
$\group{16\zeta_4,27}$ & $1791/1792$  & $5/13$ & $447/448$ & $4475/11648$ & $5/39$ & $1115/2912$ & $75/448$  & $5/117$ \\
$\group{2\zeta_4,27,25}$ &  $29/30$  & $37/80$ & $11/15$ & $259/600$ & $37/240$ & $259/1200$ & $11/60$  & $37/720$ \\ [0.5ex] 
\end{tabular}
\egroup
\end{table}

Recall the notation $\beta_{k}$ from Theorem \ref{cor-kfree} (which is the density of primes $\p$ of $K$ such that $\ord_\p(G)$ is $k$-free). Notice that, setting 
\begin{equation}\label{Akr}
A_{k,r}:=\prod_{\ell\text{ prime}}\Big(1-\frac{\ell^r-1}{(\ell-1)(\ell^{r+1}-1)\ell^{k-2}}\Big),
\end{equation}
a constant which only depends on the parameters $k\geq2$ and $r\geq1$, the infinite product in \eqref{eqcor-kfree0} is equal to 
\[ A_{k,r}\cdot\prod_{\ell\mid z}\Big(1-\frac{\ell^r-1}{(\ell-1)(\ell^{r+1}-1)\ell^{k-2}}\Big)^{-1}, \]
so $\beta_k$ can be expressed as rational multiple of $A_{k,r}$.
In Table \ref{table-constants} we show some values for  the $A_{k,r}$'s, approximated by considering only the primes $\ell$ up to $10^5$. In Tables \ref{table-Q3-kfree} and \ref{table-Q4-kfree} we provide some examples of densities $\beta_{k}$, expressed both as rational multiples of  $A_{k,r}$ and as approximated value.

\begin{table}[ht]
\caption{Examples of constants $A_{k,r}$ approximated ($\ell<10^5$)}
\label{table-constants}
\centering\scriptsize
\bgroup
\def\arraystretch{1.2}
\begin{tabular}{c||ccccccc}
$A_{k,r}$  &  $k=2$ &  $k=3$ &  $k=4$ &  $k=5$ &  $k=6$ &  $k=7$ &  $k=8$ \\ 
\hline
$r=1$ & $0.530712$  &  $0.788163$ &  $0.901926$ &  $0.953511$ & $0.977581$  & $0.989060$ & $0.994618$  \\
$r=2$ & $0.434934$  & $0.734313$  &  $0.875354$ &  $0.940597$ &  $0.971280$ & $0.985966$ &  $0.993091$ \\
$r=3$ &  $0.401045$ & $0.714103$  & $0.865118$  &  $0.935552$ & $0.968798$  & $0.984741$ & $0.992484$  \\
$r=4$ & $0.386687$  & $0.705354$  & $0.860624$  & $0.933316$  & $0.967691$  & $0.984192$ & $0.992211$  \\
$r=5$ & $0.380106$  & $0.701307$  &  $0.858528$ & $0.932267$  &  $0.967169$ & $0.983932$ & $0.992082$ 
\end{tabular}
\egroup
\end{table}

\begin{table}[ht]
\caption{Examples of densities $\beta_{k}$ over $K=\Q(\zeta_3)$}
\label{table-Q3-kfree}
\centering\scriptsize
\bgroup
\def\arraystretch{1.4}
\begin{tabular}{c||cccc}
$G$  & $\beta_{2}$ & $\beta_{3}$ & $\beta_{4}$ & $\beta_{5}$  \\ 
\hline 
$\group{2}$ & $\frac{3}{4} A_{2,1}\approx0.398$ & $\frac{121}{115} A_{3,1}\approx0.829$ & $\frac{805}{781} A_{4,1}\approx0.930$ & $\frac{5029}{4945} A_{5,1}\approx0.970$    \\
$\group{16}$ & $\frac{69}{56} A_{2,1}\approx0.654$  & $\frac{517}{460} A_{3,1}\approx0.886$ & $\frac{3325}{3124} A_{4,1}\approx0.960$ &  $\frac{20437}{19780} A_{5,1}\approx0.985$  \\
$\group{3}$ & $\frac{15}{14} A_{2,1}\approx0.569$ & $\frac{121}{115} A_{3,1}\approx0.829$ & $\frac{805}{781} A_{4,1}\approx0.930$ & $\frac{5029}{4945} A_{5,1}\approx0.970$    \\
$\group{27}$ & $\frac{55}{42} A_{2,1}\approx0.695$  & $\frac{77}{69} A_{3,1}\approx0.880$ & $\frac{2461}{2343} A_{4,1}\approx0.947$ &  $\frac{15181}{14835} A_{5,1}\approx0.976$  \\
$\group{2,3}$ &  $\frac{135}{176} A_{2,2}\approx0.334$  & $\frac{875}{814} A_{3,2}\approx0.789$ & $\frac{5989}{5750} A_{4,2}\approx0.912$ & $\frac{37823}{36994} A_{5,2}\approx0.962$   \\
$\group{16,27}$ & $\frac{899}{704} A_{2,2}\approx0.555$  & $\frac{21935}{19536} A_{3,2}\approx0.824$ & $\frac{48763}{46000} A_{4,2}\approx0.928$ &  $\frac{914711}{887856} A_{5,2}\approx0.969$  \\
$\group{2,27,25}$ & $\frac{95201}{119193} A_{2,3}$   & $\frac{105751169}{96766014} A_{3,3}$  & $\frac{524265887}{500045142} A_{4,3}$  &  $\frac{116376274169}{113496822354} A_{5,3}$   \\[-0.8ex]
& $\approx0.320$ & $\approx0.780$ & $\approx0.907$ & $\approx0.959$ 
\end{tabular}
\egroup
\end{table}

\begin{table}[ht]
\caption{Examples of densities $\beta_{k}$ over $K=\Q(\zeta_4)$}
\label{table-Q4-kfree}
\centering\scriptsize
\bgroup
\def\arraystretch{1.4}
\begin{tabular}{c||cccc}
$G$  & $\beta_{2}$ & $\beta_{3}$ & $\beta_{4}$ & $\beta_{5}$  \\ [0.5ex] 
\hline 
$\group{2}$ & $\frac{1}{4} A_{2,1}\approx0.133$ & $A_{3,1}\approx0.788$ & $A_{4,1}\approx0.902$ & $A_{5,1}\approx0.953$  \\
$\group{16}$ & $\frac{11}{8} A_{2,1}\approx0.730$ & $\frac{23}{20} A_{3,1}\approx0.906$ & $\frac{47}{44} A_{4,1}\approx0.963$ & $\frac{95}{92} A_{5,1}\approx0.985$  \\
$\group{3}$ & $\frac{3}{7} A_{2,1}\approx0.227$ & $\frac{91}{115} A_{3,1}\approx0.624$ & $\frac{709}{781} A_{4,1}\approx0.819$ & $\frac{4729}{4945} A_{5,1}\approx0.912$  \\
$\group{27}$ & $\frac{11}{21} A_{2,1}\approx0.278$ & $\frac{283}{345} A_{3,1}\approx0.647$ & $\frac{2149}{2343} A_{4,1}\approx0.827$ &  $\frac{331}{345} A_{5,1}\approx0.915$ \\
$\group{2,3}$ & $\frac{9}{176} A_{2,2}\approx0.0222$ & $\frac{329}{407} A_{3,2}\approx0.594$ & $\frac{2641}{2875} A_{4,2}\approx0.804$ & $\frac{17795}{18497} A_{5,2}\approx0.905$   \\
$\group{16,27}$ & $\frac{1073}{2112} A_{2,2}\approx0.221$ & $\frac{5501}{6512} A_{3,2}\approx0.620$ & $\frac{128873}{138000} A_{4,2}\approx0.817$ & $\frac{286741}{295952} A_{5,2}\approx0.911$  \\
$\group{2,27,25}$ & $\frac{23323}{953544} A_{2,3}$ & $\frac{79247549}{96766014} A_{3,3}$ & $\frac{3234551969}{3500315994} A_{4,3}$ & $\frac{109490052089}{113496822354} A_{5,3}$  \\[-0.8ex]
 & $\approx0.00981$  & $\approx0.585$ & $\approx0.799$ & $\approx0.903$ \\
\end{tabular}
\egroup
\end{table}

Finally, recall the notation $\gamma_{k,m}$ from Corollary \ref{cor-val} (which is the density of primes $\p$ of $K$ such that $\ord_\p(G)$ has $\ell$-adic valuation equal to $v_\ell(m)$ for every prime $\ell\mid k$). In Table \ref{table-Q5val} we provide some examples of these densities.

\begin{table}[ht]
\caption{Examples of densities $\gamma_{k,m}$ over $K=\Q(\sqrt{-5})$}
\label{table-Q5val}
\centering\scriptsize
\bgroup
\def\arraystretch{1.2}
\begin{tabular}{c||ccccccc}
$G$  & $\gamma_{6,1}$ & $\gamma_{6,2}$ &  $\gamma_{6,3}$ & $\gamma_{6,4}$ & $\gamma_{6,6}$ & $\gamma_{6,9}$ & $\gamma_{6,12}$   \\ [0.2ex] 
\hline
$\group{2}$ & $35/192$ & $35/192$ & $7/96$ & $5/24$ & $7/96$ & $7/288$ & $1/12$   \\
$\group{16}$ &  $55/96$ & $5/192$ & $11/48$ & $5/384$ & $1/96$ & $11/144$ & $1/192$   \\
$\group{3}$ & $13/48$ & $1/12$ & $1/24$ & $13/96$ & $1/6$ & $1/72$ & $1/48$  \\
$\group{27}$ & $5/16$ &  $1/4$ & $1/72$ & $5/32$ & $1/18$ & $1/216$ & $1/144$   \\
$\group{2,3}$ & $365/2912$ & $423/2912$ & $1/364$ & $101/728$ & $59/364$ & $1/1092$ & $10/91$  \\
$\group{16,27}$ & $391/1456$ & $225/2912$ & $15/364$ & $785/5824$ & $125/728$ & $5/364$ & $95/4368$  \\
$\group{2,27,25}$ & $801/6400$ & $927/6400$ & $37/28800$ & $443/3200$ & $4699/28800$ & $37/86400$ & $1591/14400$   \\ [0.2ex] 
\end{tabular}
\egroup
\end{table}
\FloatBarrier

\section*{Acknowledgments}
The author would like to thank Francesco Pappalardi for suggesting the problem and for helpful discussions (Remark \ref{rem-pappa} is due to him). Also many thanks to Antonella Perucca for her continuous support and valuable feedback. I thank the reviewers for many helpful comments, and, in particular, one of them for suggesting to use the results of \cite{TZ}, leading to an improvement of the error terms in Theorems \ref{thm-divisible} and \ref{thm-kfree}.  I thank Pieter Moree for some useful remarks. The Sage code used for the examples has been partly adapted from Sebastiano Tronto's code \emph{kummer-degrees} available on his GitHub page.


\begin{thebibliography}{30}

\bibitem{andam}
H.O.~Abdullah, A.~Ali Mustafa, F.~Pappalardi, \emph{Divisibility of reduction in groups of rational numbers II}, Int. J. Number Theory {\bf 19} (2023), no.~2, 247--260.

\bibitem{Debry}  
C.~Debry, A.~Perucca, \emph{Reductions of algebraic integers}, J. Number Theory {\bf 167} (2016), 259--283.

\bibitem{JPS} 
O.~J\"arviniemi, A.~Perucca, P.~Sgobba, \textit{Unified treatment of Artin-type problems~II}, preprint, submitted for publication (2022), \href{https://arxiv.org/abs/2211.15614}{arXiv:2211.15614}.

\bibitem{lagarias}
J.~C.~Lagarias, A.~M.~Odlyzko, \emph{Effective versions of the Chebotarev density theorem}, in Algebraic number fields: L-functions and Galois properties, Proc. Sympos., Univ. Durham, pages 409--464. Academic Press, London, 1977.


\bibitem{LMS}
H.W.~Jr.~Lenstra,  P.~Moree, P.~Stevenhagen, \emph{Character sums for primitive root densities}, Math. Proc. Cambridge Philos. Soc. {\bf 157} (2014), no.~3, 489--511.

\bibitem{moree}
P.~Moree, \emph{On primes $p$ for which $d$ divides $\ord_p(g)$}, Funct. Approx. Comment. Math. \textbf{33} (2005), 85--95.

\bibitem{MoreeArtin}
P.~Moree, \emph{Artin's primitive root conjecture -- a survey}, Integers {\bf 12} (2012), no.~6, 1305--1416.


\bibitem{murty}
V.~K.~Murty, \emph{Modular forms and the Chebotarev density theorem. II}, Analytic number theory (Kyoto, 1996), 287--308, London Math. Soc. Lecture Note Ser. {\bf 247}, Cambridge Univ. Press, Cambridge, 1997.


\bibitem{neukirch} 
J.~Neukirch, \emph{Algebraic Number Theory}, Springer, Berlin Heidelberg, 1999.

\bibitem{pappa-sf}
F.~Pappalardi, \emph{Square free values of the order function}, New York J. Math. \textbf{9} (2003), 331--344.

\bibitem{pappa}
F.~Pappalardi, \emph{Divisibility of reduction in groups of rational numbers},
Math. Comp. \textbf{84} (2015), no. 291, 385--407.

\bibitem{PeruccaKummer}  
A.~Perucca, \emph{The order of the reductions of an algebraic integer}, J. Number Theory {\bf 148} (2015), 121--136.

\bibitem{PeruccaWIN}
A.~Perucca, \emph{Reductions of algebraic integers II}, in \emph{Women in numbers Europe II}, 19--33, Assoc.\ Women Math.\ Ser., 11, Springer, Cham, 2018.

\bibitem{PS1}  
A.~Perucca, P.~Sgobba, \emph{Kummer theory for number fields and the reductions of algebraic numbers}, Int. J. Number Theory \textbf{15}  (2019), no. 8, 1617--1633.

\bibitem{PS2}  
A.~Perucca, P.~Sgobba, \emph{Kummer theory for number fields and the reductions of algebraic numbers II}, Unif. Distrib. Theory \textbf{15}  (2020), no. 1, 75--92.


\bibitem{PST}  
A.~Perucca, P.~Sgobba, S.~Tronto, \emph{The degree of Kummer extensions of number fields}, Int. J. Number Theory \textbf{17}  (2021), no. 5, 1091--1110.

\bibitem{ross}
J.~B.~Rosser,  L.~Schoenfeld, \emph{Approximate formulas for some functions of prime numbers}, Illinois J. Math. \textbf{6} (1962), no. 1, 64--94.

\bibitem{sage}
\emph{{S}ageMath, the {S}age {M}athematics {S}oftware {S}ystem ({V}ersion
  9.2)}, The Sage Developers, 2021, {\tt https://www.sagemath.org}.

\bibitem{serre}
J.-P.~Serre, \emph{Quelques applications du th{\'e}or{\`e}me de densit\'e de Chebotarev}, Inst. Hautes {\'E}tudes Sci. Publ. Math. \textbf{54} (1981), 123--201.


\bibitem{stark}
H.~M.~Stark, \emph{Some effective cases of the Brauer-Siegel theorem}, Invent. Math. {\bf 23} (1974), 135--152.

\bibitem{Stev}
P.~Stevenhagen, \emph{The correction factor in Artin's primitive root conjecture}, J. Théor. Nombres Bordeaux {\bf 15} (2003), no.~1, 383--391.

\bibitem{TZ}
J.~Thorner, A.~Zaman, \emph{A unified and improved Chebotarev density theorem}, Algebra Number Theory {\bf 13} (2019), no. 5, 1039--1068.



\bibitem{wiert}
K.~Wiertelak, \emph{On the density of some sets of primes, IV}, Acta Arith. \textbf{43} (1984), no. 2, 177--190.

\bibitem{zieg}
V.~Ziegler, \textit{On the distribution of the order of number field elements modulo prime ideals}, Unif. Distrib. Theory \textbf{1} (2006), no. 1, 65--85.

\end{thebibliography}
\end{document}